\documentclass{amsart}

\usepackage{tikz}
\usepackage{pgf}
\usetikzlibrary{arrows,shapes,snakes,automata,backgrounds,petri}
\usepackage{amsmath}
\usepackage{amssymb}
\usepackage{amsthm}
\usepackage[all]{xy}
\usepackage[latin1]{inputenc}        
\usepackage{amscd}
\usepackage{mathrsfs}
\usepackage[mathcal]{eucal}
\usepackage{float}
\author{Cinzia Bisi, Giampiero Chiaselotti}
\address{Chiaselotti Giampiero\\Dipartimento di Matematica, Universit\'a della Calabria, Via Pietro Bucci,
Cubo 30B, 87036 Arcavacata di Rende (CS), Italy.}
\email{chiaselotti@unical.it}
\address{Bisi Cinzia\\Dipartimento di Matematica, Universit\'a di Ferrara, Via Machiavelli 35, 44100, Ferrara, Italy}
\email{bsicnz@unife.it}
\email{bisi@math.unifi.it}
\title[A class of lattices]{A class of lattices and boolean functions related to a Manickam-Mikl\"os-Singhi Conjecture}
\date{\today}

\newtheorem{inizio}{Lemma}[section]

\newtheorem{proposition}[inizio]{Proposition}
\newtheorem{remark}[inizio]{Remark}

\newtheorem{definition}[inizio]{Definition}
\newtheorem{question}[inizio]{Question}
\newtheorem{o-problem}[inizio]{Open Problem}

\newtheorem*{teo-L}{Theorem}

\newtheorem*{corollary-s}{Corollary}
\theoremstyle{definition}
\newtheorem{example}[inizio]{Example}

\begin{document}

\thanks{The first author is partially supported by Progetto MIUR di
Rilevante Interesse Nazionale {\it Propriet\`a geometriche
delle variet\`a reali e complesse} and by GNSAGA - INDAM}

\abstract
The aim of this paper is to build a new family of lattices related to some combinatorial extremal sum problems, in particular to a conjecture of Manickam, Mikl\"os and Singhi. We study the fundamentals properties of such lattices and of a particular class of boolean functions defined on them. 

\endabstract

\maketitle

\section{Introduction}
Let $n,r$ be two fixed integers such that $0 \le r \le n$ and let $I_n =\{ 1,2 \cdots , n \}$. In the first part of this paper (Section \ref{Sec1}) we define a partial order $\sqsubseteq$ on the power set $\mathcal{P}(I_n)$ having the following property :
if $X, Y$ are two subsets of $I_n$ such that $X \sqsubseteq Y$, then $\sum_{i\in X}a_i \le \sum_{i\in Y}a_i$, for each $n$-multiset $\{ a_1,\cdots ,a_n\}$ of real numbers such that $a_1 \ge \cdots \ge a_r \ge 0 > a_{r+1} \ge \cdots \ge a_n$. This order defines a lattice structure on $\mathcal{P}(I_n)$ that we will denote by $(S(n,r),\sqsubseteq)$. We show as this lattice is distributive, graded (Section \ref{Sec2}), involutive (Section \ref{Sec3}), i.e. $X \sqsubseteq Y$ implies $Y^c \sqsubseteq X^c$, and we also give an algorithmic method to generate uniquely its Hasse diagram (Section \ref{Sec4}) and a recursive formula to count the number of its elements having fixed rank (Section \ref{Sec5}).

In the second part of the paper (Section \ref{Sec6}) we establish the connection between the lattice $S(n,r)$ and some combinatorial extremal sum problems related to a conjecture of Manickam, Mikl\"os and Singhi. We give an interpretation of these problems in terms of a particular class of boolean maps defined on $S(n,r)$ (Section \ref{Sec7}). 

Now we briefly summarize the historical motivations that have led us to build the lattice $S(n,r)$ and the other associated structures.

In \cite{ManMik87} the authors asked the following question:
{\it let $n$ be an integer strictly greater than 1 and $a_1, \cdots a_n$ be real numbers satisfying the property $\sum_{1=1}^n a_i \ge 0.$ We may ask: how many of the subsets of the set $\{ a_1, \cdots, a_n \}$ will have a non-negative sum?} \\
Following the notations of \cite{ManMik87}, the authors denote with $A(n)$
the minimum number of the non-negative partial sums of a sum $\sum_{i=1}^n a_i \ge 0,$ not counting the empty sum, if we take all the possible choices of the $a_i$'s. They prove (see Theorem 1 in \cite{ManMik87}) that $A(n)=2^{n-1}$ and they explain as Erd\"os, Ko and Rado investigated a question with an answer similar to this one:
{\it what is the maximum number of pairwise intersecting subsets of a $n-$elements set?} As in their case, here also the question becomes more difficult if we restrict ourselves to the $d$-subsets. More details about this remark can be find in the famous theorem of Erd\"os-Ko-Rado \cite{erd-ko-rad} (see also \cite{katona} for an easy proof of it). \\
Formally, with the introduction of the positive integer $d,$ the problem is the following. Let $1 \le d <n$ be an integer;
a function $f:I_n \to \mathbb{R}$ is called a $n-$weight function if $\sum_{x\in I_n} f(x) \ge 0.$ \\
Denote with $W_n(\mathbb{R})$ the set of all the $n-$weight functions and if $f\in W_n(\mathbb{R})$ we set
$$f^+=|\{x \in I_n: f(x) \ge 0 \}|,$$
$$\alpha (f) = |\{ Y \subseteq I_n : \sum_{y \in Y} f(y) \ge 0 \}|,$$
$$\phi (f,d) = |\{ Y \subseteq I_n : |Y|=d, \sum_{y \in Y} f(y) \ge 0 \}|,$$
 and furthermore \\
$$
\psi (n,d)=\min \{ \phi (f,d): f \in W_n(\mathbb{R}) \}.
$$
If $f$ is such that $f(1)=n-1,$ $f(2)=\cdots =f(n)=-1,$ it follows that $\psi (n,d) \le \binom{n-1}{d-1}.$ \\
In \cite{bier-man}, Bier and Manickam proved that $\psi (n,d) =\binom{n-1}{d-1}$ if $n \ge d(d-1)^d (d-2)^d +d^4$ and $\psi(n,d)=\binom{n-1}{d-1}$ if $d|n.$ 
Both the proofs use the Baranyai theorem on the factorization of complete hypergraphs \cite{baranai} (see also \cite{wilson} for a modern exposition of the theorem).
\\In \cite{ManMik87} and \cite{ManSin88} it was conjectured that $\psi (n,d) \ge \binom{n-1}{d-1}$ if $n \ge 4d.$
In \cite{ManSin88} this conjecture has been set in the more general context of the association schemes (see \cite{bailey-2004} for general references on the subject). In the sequel we will refer to this conjecture as the Manickam-Mikl\"os-Singhi (MMS) Conjecture.
This conjecture is connected with the first distribution invariant of the Johnson association scheme (see \cite{bier-man}, \cite{ManSin88}, \cite{manic-88}, \cite{manic-91}). The distribution invariants were introduced by Bier \cite{bier}, and later investigated in \cite{bier-delsarte}, \cite{manic-86}, \cite{manic-88}, \cite{ManSin88}. In \cite{ManSin88} the authors claim that this conjecture is, in some sense, dual to the theorem of Erd\"os-Ko-Rado \cite{erd-ko-rad}. Moreover, as pointed out in \cite{sriniv-98}, this conjecture settles some cases of another conjecture on multiplicative functions by Alladi, Erd\"os and Vaaler, \cite{erdos}. Partial results related to the Manickam-Mikl\"os-Singhi conjecture have been obtained also in \cite{bhatta-2003}, \cite{bhatt-thesis}, \cite{Chias02}, \cite{ChiasInfMar08}, \cite{chias-mar-2002}.\\
Now, if $1\le r \le n,$ we set:

\begin{equation} \label{eq0}
\gamma (n,r)=\min \{ \alpha (f) : f \in W_n(\mathbb{R}), f^+=r \},
\end{equation}
\begin{equation} \label{eq1}
\gamma (n,d,r)=\min \{ \phi (f,d) : f \in W_n(\mathbb{R}), f^+=r \}.
\end{equation}

The numbers $\gamma(n,d,r)$ have been introduced in \cite{Chias02} and they also have been studied in \cite{ChiasInfMar08},
in order to solve the Manickam-Mikl\"os-Singhi conjecture, because it is obviuos that:
\begin{equation} \label{eq2}
\psi (n,d)= \min \{ \gamma (n,d,r) : 1 \le r \le n \}.
\end{equation}

Therefore the complete computation of these numbers gives an answer to the MMS conjecture but this is not the purpose of this paper.

In \cite{ManMik87} it has been proved that $\gamma (n,r) \ge 2^{n-1}$ for each $r,$ and that $\gamma (n,1)=2^{n-1}$. 
\begin{question}
Is it true that $\gamma (n,r) = 2^{n-1}$ for each $r$?
\end{question}
This is true if, for each $r$ such that $1 \le r \le n,$ we can find a function $f \in W_n(\mathbb{R})$ with $\alpha (f)=2^{n-1}$.

Let us observe now  that when we have a $n$-weight function $f$, the standard ways to produce $n$-subsets on which $f$ takes non-negative values are the following :

(i) if we know that $X$ and $Y$ are two subsets of $I_n$ such that $\sum_{x \in X}f(x) \ge 0$ and $\sum_{x\in X}f(x) \le \sum_{x\in Y}f(x)$, then also $\sum_{x\in Y}f(x) \ge 0$ ({\it monotone} property);

(ii) if we know that $\sum_{x\in X}f(x) < 0$, then $\sum_{x\in X^c}f(x) \ge 0$ ({\it complementary} property).

Then we ask:

A) Is it possible to axiomatize the  properties (i) and (ii) in some type of abstract structure in such a way that the sum extremal problems upon described become particular extremal problems of more general problems?

B) In such abstract structure can we find unexpected links with other theories which help us to solve these sum extremal problems?

C) Is it possible to define an algorithmic strategy in such abstract structure to approach these sum extremal problems in a deterministic way?


In this paper we show that the answer to all the previous questions is affirmative.

We define a partial order $\sqsubseteq$ on the subsets of $I_n$ such that if $X$ and $Y$ are two subsets with $X \sqsubseteq Y$, then $\sum_{x \in X} f(x) \le \sum_{x \in Y} f(x)$, for each $n$-weight function $f$. In the first part of this paper, we study the fundamental properties of this order (see Section \ref{Sec1}, \ref{Sec2}, \ref{Sec3}).

The attempt of computing the numbers in (\ref{eq0}) and (\ref{eq1}) for each $n,d,r$ gives us the idea to construct two types of lattices, denoted by  $S(n,r)$ and $S(n,d,r)$, and to transform the problem of computing the numbers $\gamma(n,r)$ and $\gamma (n,d,r)$ into the problem of computing a minimal cardinality on a family of posets. \\
This way of consider the problem has many advantages. For example, when we try to prove that $\gamma(n,r)$ is not greater than $2^{n-1}$ for each $r$, we need to build a particular $n$-weight function $f$ with $f^+ = r$ such that $\alpha(f) \le 2^{n-1}$. In general, this leads to examine a certain number of inequalities, and if such number is big the determination of $f$ can be difficult. The case of $\gamma(n,d,r)$ is similar and, obviously, more difficult. In general, if our aim is to prove that $\gamma(n,r) \le T$ (or $\gamma(n,d,r) \le T$), for some number $T$, it is natural to ask : is it possible to determine a minimal number of inequalities which leads us to find a $n$-weight function $f$ with $f^+ = r$, such that $\alpha(f) \le T$ (or $\phi(f,d) \le T$)? If we identify (in some sense) each $n$-weight function with a particular type of boolean map defined on the lattice of the subsets of $I_n$, with the order $\sqsubseteq$, the number of these maps will be finite, and even if such number is large, the study of the properties of the lattice could lead to examine a more restricted class of these maps, that lends itself to a simpler study.


To better understand what we have just asserted, let us consider an example.
\begin{example}
Let $n=8$ and $r=d=5$. Let $f$ be the following $8$-weight function with $f^+ = 5$ :

\begin{equation}\label{forma-f}
   \begin{array}{ccccccccc}
    \tilde{5} & \tilde{4} & \tilde{3} & \tilde{2} & \tilde{1} & \overline{1} & \overline{2} & \overline{3} &  \\
    \downarrow & \downarrow & \downarrow & \downarrow & \downarrow & \downarrow & \downarrow & \downarrow & \\
    \frac{1}{5} & \frac{1}{5} &  \frac{1}{5} & \frac{1}{5} & \frac{1}{5} & -\frac{1}{3} & -\frac{1}{3} & -\frac{1}{3} &\\
  \end{array}
\end{equation}

(Upon we have written $I_8$ in the form $\{\tilde{5}, \tilde{4}, \tilde{3}, \tilde{2}, \tilde{1}, \overline{1},{\,} \overline{2}, \overline{3} \}$, and in the following we write, for example, the $5$-tuple $\tilde{5}\tilde{3}\tilde{1}\overline{2} \,\, \overline{3}$ as $531|23$).
It follows easily that $\phi(f,5)=\binom{5}{5} + 3\binom{5}{4}=16$. Therefore, by (\ref{eq1}) we have $\gamma(8,5,5)\le 16$.
To prove that also the inverse inequality holds, we fix an arbitrary $8$-weight function $f$ with $f^+ = 5$ and we prove that it has always at least 16 $5$-tuples on which it takes a non-negative value. Then, if we consider the $5$-tuple $4321|3$, it is easy to see that its non-negativity implies also the non-negativity of $16$ other $5$-tuples (included itself). This means that in the sublattice $S(8,5,5)$ of the lattice $S(8,5)$ the element $4321|3$ spans an up-set having $16$ elements. Then we say that the element $4321|3$ has {\it positive weight} $16$. Therefore we can assume that $f$ has negative sum on $4321|3$. Since $f$ is a weight-function, it takes then non-negative sum on the complementary $3$-tuple $5|12$. It is easily seen then that the non-negativity of $5|12$ produces exactly the non-negativity of $15$ other $5$-tuples. Let us consider now the $5$-tuple $4321|1$ (which is not included in the non-negative $5$-tuples above described). If $f$ takes non-negative sum on $4321|1$, we have produced exactly $16$ other $5$-tuples with non-negative sum for $f$. If $f$ takes negative sum on $4321|1$, then it must take non-negative sum on the complementary $3$-tuple $5|23$, and this produces $16$ other $5$-tuples having non-negative sum and different from the previous $15$ $5$-tuples; therefore we obtain in this case $(15+16=31)$ $5$-tuples having non-negative sum. This shows that $\gamma(8,5,5)=16$.
\end{example}

In the previous computation of the number $\gamma(8,5,5)$ we can note as the only properties that we have used are the monotone and the complementary properties. Then, to define an algorithmic procedure which holds for each value of $n$, $d$ and $r$, we need an order structure which includes all the subsets of $I_n$, and not only those with $d$ elements, since their set is not closed with respect to complementary operation.

In this paper we concentrate our attention on the numbers $\gamma(n,r)$ and we will approach the study of the $\gamma(n,d,r)$'s in subsequent papers. Here we build a formal context which makes sense out of what is said above and also out of the question raised in \cite{ManMik87}: ``What is the structure of the constructions giving this extremal value?". We show also that the problems above described can be considered as problems related to a particular class of boolean functions defined on our order structures. The properties of these boolean functions generalize the essential properties of the weight-functions, i.e. the order preserving and the complementary property.

In Section \ref{Sec7}, we state two open problems, which are substantially two statements of representation theorems. If the answer to these problems will be affirmative, the problem of determining the numbers $\gamma(n,r)$ and $\gamma(n,d,r)$ will be equivalent to the problem of determining the minimum number of elements which have value $1$ for a particular type of boolean functions. The advantage of this approach consists in the possibility to use the results of the combinatorial lattice theory.

To conclude, we believe that the study of the extremal sum problems settled in \cite{ManMik87} and in \cite{ManSin88} (among which the Manickam-Mikl\"os-Singhi Conjecture), in the setting of the lattices $S(n,r),$ is interesting because it can lead to unexpected links among the combinatorial theory of the lattices, the theory of association schemes (good references for the link between the association schemes and the extremal problem on the non-negative sums of real numbers are \cite{bier-man}, \cite{manic-86}, \cite{manic-88}, \cite{manic-91}, \cite{ManSin88}), the transversal theory (see \cite{ChiasInfMar08}, in which the Hall theorem has been used for computing the $\gamma(n,d,r)$'s with $n=2d+2$ and $n-r=3$) and the theory of boolean functions defined on particular classes of posets.
\\ For example this lattice structure $S(n,r)$ could be useful in the computation of the higher order distribution invariants of the Johnson association scheme, \cite{ManSin88}, \cite{bier-man}: this will be the main object of a forthcoming investigation.

In this paper we adopt the  classical terminology and notations usually used in the context of the partially ordered sets  (see \cite{dav-pri-2002} and \cite{stanley-vol1} for the general aspects on this subject). In particular, if $(P, \le )$ is a poset and $Q \subseteq P$, we set $\downarrow Q = \{y\in P\,\,\, |\,\,\, (\exists \,\,\, x\in Q)\,\,\, y \le x\}$, $\uparrow Q = \{y\in P\,\,\, |\,\,\, (\exists \,\,\, x\in Q)\,\,\, y \ge x\}$, and  $\downarrow \{x\} = \downarrow x$, $\uparrow \{x\} = \uparrow x$, for each $x\in P$. A subset $Q$ of $P$ is said to be a {\it down-set} (or {\it up-set}) of $P$ if $Q=\downarrow Q$ (or $Q=\uparrow Q$).

\section{The Lattice $S(n,r)$ and its Sublattice $S(n,d,r)$}\label{Slattices} \label{Sec1}
Let $n$ and $r$ be two fixed integers such that $0\le r \le n.$ We denote with $A(n,r)$ an alphabet composed by
the following $(n+1)$ formal symbols: $\tilde{1}, \cdots, \tilde{r},\,\,\,0^§,\,\,\, \overline{1}, \cdots, \overline{n-r}.$ We introduce on $A(n,r)$ the following total order:
\begin{equation}\label{totalorder}
\overline{n-r} \prec \cdots \prec \overline{2} \prec \overline{1} \prec 0^§ \prec \tilde {1} \prec \tilde{2}\prec \cdots \prec \tilde{r},
\end{equation}
where $\overline{n-r}$ is the minimal element and $\tilde{r}$ is the maximal element in this chain.
If $i,j \in A(n,r),$ then we write $i \preceq j$ for $i=j$ or $i \prec j$;  $i \curlywedge j$ for the minimum and $i \curlyvee j$ for the maximum between $i$ and $j$ with respect to $\preceq$; $i \vdash j$ if $j$ covers $i$ with respect to $\preceq$ (i.e. if $i \prec j$ and if there does not exist $l\in A(n,r)$ such that $i\prec l \prec j$); $i \nvdash j$ if $j$ does not cover $i$ with respect to $\preceq$;  $j \succ i$ for $i \prec j$;  $j \succeq i$ for $i \preceq j$.
 We set $(\mathcal{C}(n,r), \sqsubseteq)$ the $n$-fold cartesian product poset $A(n,r)^n$. An arbitrary element of $\mathcal{C}(n,r)$ can be identified with an $n$-string $t_1\cdots t_n$ where $t_i \in A(n,r)$ for all $i=1, \cdots, n.$ Therefore, if $t_1\cdots t_n$ and $s_1\cdots s_n$ are two strings of $\mathcal{C}(n,r)$, we have
 \begin{equation*}
 t_1\cdots t_n \sqsubseteq s_1\cdots s_n \Longleftrightarrow t_1 \preceq s_1, \cdots , t_n \preceq s_n.
\end{equation*}

   We introduce now a particular subset $S(n,r)$ of $\mathcal{C}(n,r)$.

A string of $S(n,r)$ is constructed as follows: it is a formal expression of the following type
\begin{equation} \label{stringa}
i_1 \cdots  i_r \,\,\, | \,\,\,  j_1 \cdots  j_{n-r},
\end{equation}
where $i_1, \cdots, i_r \in \{ \tilde{1}, \cdots, \tilde{r}, 0^§ \},$ $j_1, \cdots, j_{n-r} \in \{ \overline{1}, \cdots, \overline{n-r}, 0^§ \}$ and where the choice of the symbols has to respect the following two rules, see (\ref{prop3}) and (\ref{prop4}):\\
\begin{equation} \label{prop3}
i_1 \succeq \cdots \succeq i_r \succeq 0^§ \succeq j_1 \succeq \cdots \succeq j_{n-r};
\end{equation}
furthermore, if we set
\begin{equation} \label{pindex}
p= \left\{ \begin{array}{ll}
      \textrm{max} \{ l: l \in \{1, \cdots, r \} \,\,\, \textrm{with} \,\,\, i_l \succ 0^§ \} & \textrm{if}\,\, l \,\, \textrm{exists} \\
              0 &  \textrm{otherwise} \\
              \end{array} \right.
\end{equation}
and
\begin{equation} \label{qindex}
q= \left\{ \begin{array}{ll}
      \textrm{min} \{ l: l \in \{1, \cdots, n-r \} \,\,\, \textrm{with} \,\,\, j_l \prec 0^§ \} & \textrm{if}\,\, l \,\, \textrm{exists} \\
              n-r+1 &  \textrm{otherwise} \\
              \end{array} \right.
\end{equation}
then
\begin{equation} \label{prop4}
\begin{array}{ll}
i_1 \succ \cdots \succ i_p \succ 0^§  \,\,\,\, , & i_{p+1}=\cdots =i_r=0^§\\
j_1= \cdots =j_{q-1}=0^§ \,\,\,\, , & 0^§ \succ j_q \succ \cdots \succ j_{n-r}. \\
\end{array}
\end{equation}
If $p=0$ we assume that $i_1= \cdots =i_r=0^§$ and the condition $i_1 \succ \cdots \succ i_p \succ 0^§$ is empty; if $q=(n-r+1),$ we assume that $j_1= \cdots =j_{n-r}=0^§$ and the condition $0^§ \succ j_q \succ \cdots \succ j_{n-r}$ is empty.
The formal symbols which appear in (\ref{stringa}) will be written without  $\;\tilde{}\;$, $\;\bar{}\;$, and $\;^§\;$; the vertical bar $|$ in (\ref{stringa}) will indicate that the symbols on the left of $|$ are in
$\{ \tilde{1}, \cdots, \tilde{r}, 0^§ \}$ and the symbols on the right of $|$ are in $\{ 0^§, \overline{1}, \cdots, \overline{n-r} \}.$
\begin{example}
\begin{itemize}
\item[a)] If $n=3$ and $r=2,$ then $A(3,2)= \{ \tilde{2} \succ \tilde{1} \succ 0^§ \succ \overline{1} \}$ and
$S(3,2)= \{ 21|0, \,\,\, 21|1,\,\,\, 10|0,\,\,\, 20|0,\,\,\, 10|1,\,\,\, 20|1,\,\,\, 00|1,\,\,\, 00|0 \}.$\\
\item[b)] If $n=3$ and $r=0,$ then $A(3,0)=\{ 0^§ \succ \overline{1} \succ \overline{2} \succ \overline{3} \}$
and $S(3,0)=\{ |123,\,\,\,|023,\,\,\,|013,\,\,\,|012,\,\,\,|003,\,\,\,|002,\,\,\,|001,\,\,\,|000 \}.$ \\
\item[c)] If $n=0$ and $r=0,$ then $S(0,0)$ will be identified with a singleton $\Gamma$ corresponding
to $|$ without symbols. \\
\end{itemize}
\end{example}
In the sequel $S(n,r)$ will be considered as sub-poset of $\mathcal{C}(n,r)$ with the induced order from $\sqsubseteq$ after the restriction to $S(n,r).$
Therefore, if $w=i_1 \cdots i_r \,\,\,| \,\,\,j_1 \cdots j_{n-r}$ and $w'=i_1^{'} \cdots i_r^{'} \,\,\,| \,\,\,j_1^{'} \cdots j_{n-r}^{'}$ are two strings in $S(n,r),$ by definition of induced order we have
\begin{equation*}
w \sqsubseteq w' \Longleftrightarrow i_1 \preceq i_1^{'}, \cdots , i_r \preceq i_r^{'},\,\,\, j_1 \preceq j_1^{'}, \cdots ,j_{n-r} \preceq j_{n-r}^{'}.
\end{equation*}
As it is well known, $(\mathcal{C}(n,r), \sqsubseteq)$ is a distributive lattice whose binary operations of {\it inf} and {\it sup} are given respectively by $(t_1\cdots t_n) \wedge (s_1\cdots s_n) = (t_1\curlywedge s_1)\cdots (t_n\curlywedge s_n)$, and $(t_1\cdots t_n) \vee (s_1\cdots s_n) = (t_1\curlyvee s_1)\cdots (t_n\curlyvee s_n)$.

\begin{example}
If $n=7$ and $r=4,$ and if $w_1=4310|023,$ and $w_2=2100|012,$  are two elements of $S(7,4)$, then $w_1 \wedge w_2=2100|023,$ and $w_1 \vee w_2=4310|012.$ \\
\end{example}

In general, if $w_1, w_2 \in S(n,r)$ it is immediate to verify that $w_1\wedge w_2 \in S(n,r)$ and $w_1\vee w_2 \in S(n,r)$. Therefore the following proposition holds:

\begin{proposition} \label{distributS(n,r)}
$(S(n,r), \sqsubseteq)$ is a distributive lattice.
\end{proposition}
\begin{proof}

$(\mathcal{C}(n,r), \sqsubseteq)$ is a distributive lattice and $S(n,r)$ is closed with respect to $\wedge$ and $\vee$. Hence $S(n,r)$ is a distributive sublattice of $\mathcal{C}(n,r)$.\end{proof}

\begin{definition}
If $w_1, w_2 \in S(n,r)$, then
\begin{itemize}
\item[i)] $w_1 \sqsubset w_2$ if $w_1 \sqsubseteq w_2$ and $w_1 \neq w_2$;\\
\item[ii)] $w_1\models w_2$ if $w_2$ covers $w_1$ with respect to the order $\sqsubseteq$ in $S(n,r)$ (i.e. if $w_1 \sqsubset w_2$ and there doesn't exist $w\in S(n,r)$ such that $w_1 \sqsubset w \sqsubset w_2$);\\
\item[iii)] $w_1\nvDash w_2$ if $w_2$ does not cover $w_1$ with respect to the order $\sqsubseteq$ in $S(n,r)$.
\end{itemize}
\end{definition}
\begin{remark}
The minimal element of $S(n,r)$ is the string  $0\cdots 0|12 \cdots (n-r)$ and the maximal element is $r(r-1) \cdots 1|0\cdots 0.$ Sometimes they are denoted respectively with $\hat{0}$ and $\hat{1}.$
\end{remark}
If $w$ is a string in $S(n,r)$ in the form (\ref{stringa}) with $p$ and $q$ defined as in (\ref{pindex}) and (\ref{qindex}) (and (\ref{prop3}) and (\ref{prop4}) hold), we set:
\begin{equation*}
w^*=\{ i_1, \cdots, i_p, j_q, \cdots, j_{n-r} \}.
\end{equation*}
For example, if $w=4310|013 \in S(7,4),$ then $w^*=\{ \tilde{1}, \tilde{3}, \tilde{4}, \overline{1}, \overline{3} \}.$ In particular, if $w=0 \cdots 0|0 \cdots 0$ then $w^*=\emptyset.$ \\
It stays defined a bijective map
$$
*: w \in S(n,r) \mapsto w^* \in \mathcal{P} (A(n,r) \setminus \{ 0^§ \} ).
$$
On the contrary, if $B \in \mathcal{P} (A(n,r)\setminus \{ 0^§ \}),$ then $B=B_1 \cup B_2$ (with $B_1 \cap B_2=\emptyset$) where
$B_1=\{ i_1, \cdots i_p \} \subseteq \{ \tilde{1}, \cdots, \tilde{r} \}$ or $B_1=\emptyset$ and $B_2=\{ j_q, \cdots, j_{n-r} \} \subseteq \{\overline{1}, \cdots, \overline{n-r} \}$ or $B_2=\emptyset,$ for some integer $p$ and $q$ such that $1\le p \le r,$
$1 \le q \le n-r,$ with $i_1 \succ \cdots \succ i_p \succ 0^§ \succ j_q \cdots \succ j_{n-r}.$ \\
We will set
$$
\overline{B}= \left\{ \begin{array}{lll}

                       i_1 \cdots i_p 0 \cdots 0|0 \cdots 0 j_q \cdots j_{n-r} & B_1\neq \emptyset & B_2 \neq \emptyset \\
                       i_1 \cdots i_p 0 \cdots 0|0 \cdots 0 0 \cdots 0         & B_1 \neq \emptyset & B_2 =\emptyset \\
                       0 \cdots   0 0 \cdots 0|0 \cdots 0 j_q \cdots j_{n-r}   & B_1=\emptyset & B_2 \neq \emptyset \\
                       0 \cdots   0 0 \cdots 0|0 \cdots 0 0   \cdots 0         & B_1=\emptyset & B_2=\emptyset
                       \end{array}
                       \right.
$$
It stays defined a map:
$$
\bar{}:B\in \mathcal{P}(A(n,r) \setminus \{ 0^§ \}) \mapsto \overline{B} \in S(n,r),
$$
which is the inverse of the previous map $*,$ indeed:
$\overline{w^*}=w$ and $(\overline{B})^*=B,$
for each $B \in \mathcal{P}(A(n,r)\setminus \{ 0^§ \})$ and for each $w\in S(n,r).$  \\
For example, if $B=\{ \tilde{1}, \overline{1} \} \in \mathcal{P}(A(7,5) \setminus \{ 0^§ \})$ then $\overline{B}=10000|01$ is the corresponding string in $S(7,5)$.\\
We define now the following operations on $S(n,r):$ \\
if $w_1,w_2 \in S(n,r),$ we will set \\
\begin{itemize}
\item[i)] $w_1 \sqcup w_2 =\overline{w_1^* \cup w_2^*};$ \\
\item[ii)] $w_1 \sqcap w_2= \overline{w_1^* \cap w_2^*};$ \\
\item[iii)] $w_1^c=\overline{(w_1^*)^{\pi}};$
\end{itemize}
where $(w_1^*)^{\pi}$ means the complement of $w_1^*$ in $A(n,r) \setminus \{ 0^§ \}.$\\
For example, if $w_1=4310|001$ and $w_2=2000|012$ are two strings of $S(7,4),$ then \\
$w_1 \sqcup w_2= \overline{w_1^* \cup w_2^*}=\overline{ \{\tilde{1},\tilde{3},\tilde{4}, \overline{1} \} \cup \{\tilde{2},\overline{1},\overline{2} \} }=\overline{ \{ \tilde{1},\tilde{2},\tilde{3},\tilde{4}, \overline{1}, \overline{2} \} }=4321|012;$ \\
$w_1 \sqcap w_2= \overline{w_1^* \cap w_2^*}=\overline{\{ \tilde{1}, \tilde{3}, \tilde{4}, \overline{1} \} \cap \{\tilde{2}, \overline{1}, \overline{2} \} }=\overline{\{ \overline{1} \}}=0000|001;$ \\
$w_1^c=\overline{(w_1^*)^{\pi}}=\overline{ \{ \tilde{1}, \tilde{3}, \tilde{4}, \overline{1} \}^{\pi}}=\overline{ \{ \tilde{2},\overline{2},\overline{3} \}} =2000|023,$ and $w_2^c=\overline{(w_2^*)^{\pi}}=\overline{ \{\tilde{2}, \overline{1}, \overline{2} \}^{\pi} }=\overline{ \{ \tilde{1},\tilde{3},\tilde{4},\overline{3} \}}=4310|003.$ \\
\begin{remark} \label{InvolutionPropertyS(n,r)}
By previous definitions, it is immediate to verify that $(w^c)^c=w$ for all $w\in S(n,r).$
\end{remark}
Now suppose that $1 \le r \le n$ and that $d$ is a fixed integer such that $1 \le d \le n.$ We denote with $S(n,d,r)$ the set of all the strings of $S(n,r)$ such that in their form $(\ref{stringa})$ contain exactly $d$ symbols of the alphabet $A(n,r)$ different from $0^§$.
\begin{proposition}
$S(n,d,r)$ is a distributive sublattice of $S(n,r).$
\end{proposition}
\begin{proof}
It is sufficient to prove that, given $w_1,w_2 \in S(n,d,r),$ it holds that $w_1 \wedge w_2 \in S(n,d,r)$ and $w_1 \vee w_2 \in S(n,d,r).$ \\
Then we have that:
$$w_1 = i_1 \cdots i_k 0\cdots 0|0 \cdots 0 i_q \cdots i_{n-r} $$
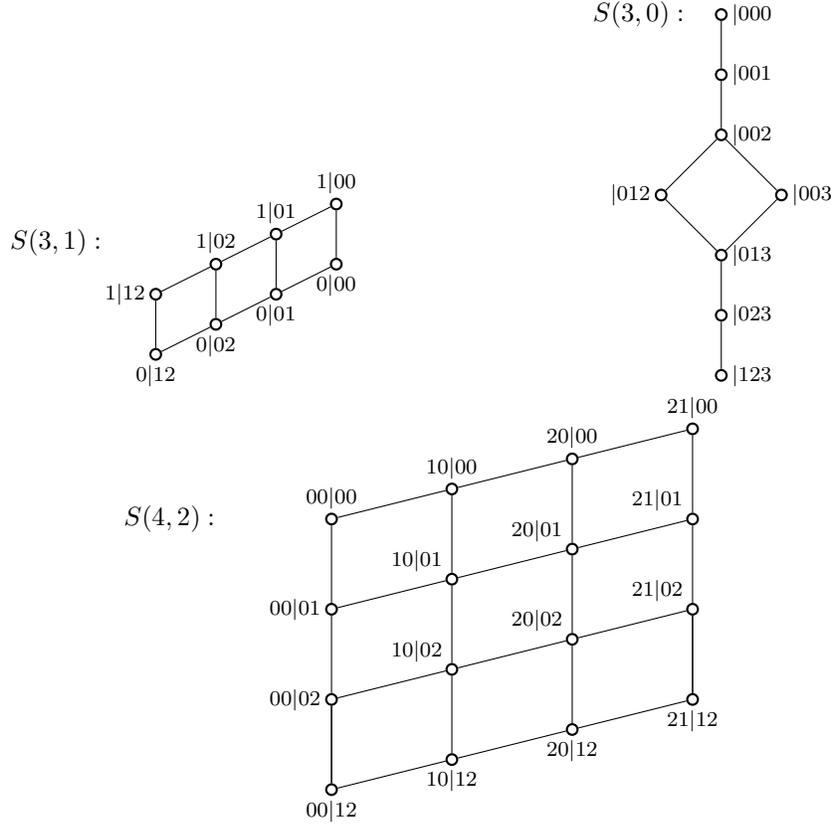
$$w_2 = i_1^{'} \cdots i_p^{'} 0\cdots 0|0 \cdots 0 i_s^{'} \cdots i_{n-r}^{'}$$
with $i_1, \cdots, i_k,$ $k$ symbols different from $0^§$ and $i_q, \cdots, i_{n-r},$ $(d-k)$ symbols different from $0^§;$ $i_1' \cdots i_p',$ $p$ symbols different from $0^§$ and $i_s', \cdots, i_{n-r}',$ $(d-p)$ symbols different from $0^§$. \\
If $k=p,$ then $(d-k)=(d-p)$ and hence $w_1 \wedge w_2$ and $w_1 \vee w_2$ have exactly $d$ symbols different from $0^§$.\\
If $k>p,$ then $w_1 \vee w_2$ has $k$ symbols different from $0^§$ on the left of $|$ and since $(d-k) < (d-p),$ $w_1 \vee w_2$ has $(d-k)$ symbols different from $0^§$ on the right of $|;$ hence $w_1 \vee w_2$ has exactly $d$ symbols different from $0^§$.\\
On the other hand, $w_1 \wedge w_2$ has $p$ symbols different from $0^§$ on the left of $|,$ and since $(d-k) < (d-p),$ $w_1 \wedge w_2$ has $(d-p)$ symbols different from $0^§$ on the right of $|;$ hence $w_1 \wedge w_2$ has exactly $d$ symbols different from $0^§$.\\
Analogously if $k <p.$
\end{proof}

\begin{remark}
Let us observe that the map $*$ induces a bijection between the power set with $d$ elements of $A(n,r) \setminus \{ 0^§ \},$ denoted with $\mathcal{P}_d (A(n,r)\setminus \{ 0^§ \} )$ and the distributive lattice $S(n,d,r).$
\end{remark}

\section{Fundamental Properties of the Lattice $S(n,r)$} \label{Sec2}
The Hasse diagrams of the lattices $S(n,r)$ for the first values of $n$ and $r$ are the following:

\begin{center}

\begin{tikzpicture}
 [inner sep=0.5mm,
 place/.style={circle,draw=black!100,fill=white!100,thick},scale=0.4]
 \path (0,0)node(0) [place,label=180:{$S(0,0): \quad $},label=0 : {\footnotesize$ $}]{};

 \end{tikzpicture}
 \hfil

 \begin{tikzpicture}
 [inner sep=0.5mm,
 place/.style={circle,draw=black!100,fill=white!100,thick},scale=0.4]
 \path (0,0)node(A) [place,label=0:{\footnotesize $0|$}]{}
 (0,2)node (B)[place,label=180:{$S(1,1): \quad $},label=0 : {\footnotesize$1|$}]{};

 \draw (A)--(B);

 \end{tikzpicture}
 \hfil
\begin{tikzpicture}
 [inner sep=0.5mm,
 place/.style={circle,draw=black!100,fill=white!100,thick},scale=0.4]
 \path (0,0)node(A) [place,label=0:{\footnotesize$|0$}]{}
 (0,2)node (B)[place,label=180:{$S(1,0): \quad $},label=0:{\footnotesize$|1$}]{};
 \draw (A)--(B);
 \end{tikzpicture}
 \\
 \vspace{0.5cm}
  \begin{tikzpicture}
 [inner sep=0.5mm,
 place/.style={circle,draw=black!100,fill=white!100,thick},scale=0.4]
 \path (0,0)node(A) [place,label=0:{\footnotesize$00|$}]{}
 (0,2)node(C) [place,label=0:{\footnotesize$10|$}]{}
 (0,4)node(D) [place,label=0:{\footnotesize$20|$}]{}

 (0,6)node (B)[place,label=180:{$S(2,2): \quad $},label=0:{\footnotesize$21|$}]{};

 \draw (A)--(C)--(D)--(B);

 \end{tikzpicture}
  \hspace{1cm}
  \hfil
 \begin{tikzpicture}
 [inner sep=0.5mm,
 place/.style={circle,draw=black!100,fill=white!100,thick},scale=0.4]
 \path (0,0)node(A) [place,label=270:{\footnotesize$0|1$}]{}
 (2,2)node(C) [place,label=0:{\footnotesize$1|1$}]{}
 (-2,2)node(D) [place,label=180:{\footnotesize$0|0$}]{}

 (0,4)node (B)[place,label=180:{$S(2,1): \qquad \qquad $},label=90:{\footnotesize$1|0$}]{};

 \draw (A)--(C);
 \draw (A)--(D);
 \draw (B)--(C);
 \draw (B)--(D);

 \end{tikzpicture}
 \hfil
  \hspace{1cm}
 \begin{tikzpicture}
 [inner sep=0.5mm,
 place/.style={circle,draw=black!100,fill=white!100,thick},scale=0.4]
 \path (0,0)node(A) [place,label=0:{\footnotesize$|12$}]{}
 (0,2)node(C) [place,label=0:{\footnotesize$|02$}]{}
 (0,4)node(D) [place,label=0:{\footnotesize$|01$}]{}

 (0,6)node (B)[place,label=180:{$S(2,0): \quad $},label=0:{\footnotesize$|00$}]{};

 \draw (A)--(C)--(D)--(B);

 \end{tikzpicture}
 \\
 \vspace{1cm}
  \begin{tikzpicture}
 [inner sep=0.5mm,
 place/.style={circle,draw=black!100,fill=white!100,thick},scale=0.4]
 \path (0,0)node(A) [place,label=0:{\footnotesize$000|$}]{}
 (0,2)node(C) [place,label=0:{\footnotesize$100|$}]{}
 (0,4)node(D) [place,label=0:{\footnotesize$200|$}]{}
 (2,6)node(E) [place,label=0:{\footnotesize$210|$}]{}
 (0,8)node(F) [place,label=0:{\footnotesize$310|$}]{}
 (0,10)node(G) [place,label=0:{\footnotesize$320|$}]{}
 (-2,6)node(H) [place,label=180:{\footnotesize$300|$}]{}
 (0,12)node (B)[place,label=180:{$S(3,3): \quad $},label=0:{\footnotesize$321|$}]{};

 \draw (A)--(C)--(D)--(E)--(F)--(G)--(B);
 \draw (D)--(H)--(F);

 \end{tikzpicture}
 \hspace{2cm}
 \hfil
  \begin{tikzpicture}
 [inner sep=0.5mm,
 place/.style={circle,draw=black!100,fill=white!100,thick},scale=0.4]
\path (0,0)node(A) [place,label=270:{\footnotesize$00|1$}]{}
 (0,2)node(B) [place,label=180:{\footnotesize$00|0$}]{}
 (2,1)node(D) [place,label=270:{\footnotesize$10|1$}]{}
 (4,2)node(E) [place,label=270:{\footnotesize$20|1$}]{}
 (4,4)node(F) [place,label=90:{\footnotesize$20|0$}]{}
 (6,3)node(G) [place,label=270:{\footnotesize$21|1$}]{}
 (6,5)node(H) [place,label=90:{\footnotesize$21|0$}]{}
 (2,3)node (C)[place,label=135:{$S(3,2): \qquad \qquad $},label=90:{\footnotesize$10|0$}]{};

 \draw (A)--(B)--(C)--(F)--(H);
 \draw (A)--(D)--(E)--(G)--(H);
 \draw (D)--(C);
 \draw (E)--(F);
 \end{tikzpicture}
 \\
 \vspace{1cm}
 \begin{tikzpicture}
 [inner sep=0.5mm,
 place/.style={circle,draw=black!100,fill=white!100,thick},scale=0.4]
\path (0,0)node(A) [place,label=270:{\footnotesize$0|12$}]{}
 (0,2)node(B) [place,label=180:{\footnotesize$1|12$}]{}
 (2,1)node(D) [place,label=270:{\footnotesize$0|02$}]{}
 (4,2)node(E) [place,label=270:{\footnotesize$0|01$}]{}
 (4,4)node(F) [place,label=90:{\footnotesize$1|01$}]{}
 (6,3)node(G) [place,label=270:{\footnotesize$0|00$}]{}
 (6,5)node(H) [place,label=90:{\footnotesize$1|00$}]{}
 (2,3)node (C)[place,label=135:{$S(3,1):  \qquad \qquad $},label=90:{\footnotesize$1|02$}]{};

 \draw (A)--(B)--(C)--(F)--(H);
 \draw (A)--(D)--(E)--(G)--(H);
 \draw (D)--(C);
 \draw (E)--(F);
 \end{tikzpicture}
 \hfil
 \hspace{2cm}
  \begin{tikzpicture}
 [inner sep=0.5mm,
 place/.style={circle,draw=black!100,fill=white!100,thick},scale=0.4]
 \path (0,0)node(A) [place,label=0:{\footnotesize$|123$}]{}
 (0,2)node(C) [place,label=0:{\footnotesize$|023$}]{}
 (0,4)node(D) [place,label=0:{\footnotesize$|013$}]{}
 (2,6)node(E) [place,label=0:{\footnotesize$|003$}]{}
 (0,8)node(F) [place,label=0:{\footnotesize$|002$}]{}
 (0,10)node(G) [place,label=0:{\footnotesize$|001$}]{}
 (-2,6)node(H) [place,label=180:{\footnotesize$|012$}]{}
 (0,12)node (B)[place,label=180:{$S(3,0): \quad $},label=0:{\footnotesize$|000$}]{};

 \draw (A)--(C)--(D)--(E)--(F)--(G)--(B);
 \draw (D)--(H)--(F);

  \end{tikzpicture}

 \vspace{0.1mm}

 \begin{tikzpicture}
 [inner sep=0.5mm,
 place/.style={circle,draw=black!100,fill=white!100,thick},scale=0.4]
\path (0,0)node(A) [place,label=270:{\footnotesize$00|12$}]{}
 (0,3)node(B) [place,label=180:{\footnotesize$00|02$}]{}
 (4,1)node(D) [place,label=270:{\footnotesize$10|12$}]{}
 (8,2)node(E) [place,label=270:{\footnotesize$20|12$}]{}
 (8,5)node(F) [place,label=145:{\footnotesize$20|02$}]{}
 (12,3)node(G) [place,label=270:{\footnotesize$21|12$}]{}
 (12,6)node(H) [place,label=145:{\footnotesize$21|02$}]{}
 (0,6)node(L) [place,label=180:{\footnotesize$00|01$}]{}
 (0,9)node(M) [place,label=180:{$ S(4,2): \qquad \qquad $},label=90:{\footnotesize$00|00$}]{}
 (4,7)node(N) [place,label=145:{\footnotesize$10|01$}]{}
 (4,10)node(O) [place,label=90:{\footnotesize$10|00$}]{}
  (8,8)node(P) [place,label=145:{\footnotesize$20|01$}]{}
 (8,11)node(Q) [place,label=90:{\footnotesize$20|00$}]{}
  (12,9)node(R) [place,label=145:{\footnotesize$21|01$}]{}
 (12,12)node(S) [place,label=90:{\footnotesize$21|00$}]{}

 (4,4)node (C)[place,label=145:{\footnotesize$10|02$}]{};

 \draw (A)--(B)--(C)--(F)--(H);
 \draw (A)--(D)--(E)--(G)--(H);
 \draw (A)--(B)--(L)--(M);
 \draw (F)--(P)--(Q);
 \draw (C)--(N)--(O);
 \draw (G)--(H)--(R)--(S);
 \draw (L)--(N)--(P)--(R);
 \draw (M)--(O)--(Q)--(S);
 \draw (D)--(C);
 \draw (E)--(F);
 \end{tikzpicture}

 \end{center}


\begin{proposition} \label{CartesianIsomorphism}
If $0\le r \le n,$ then
$S(n,r) \cong S(r,r) \times S(n-r,0).$
\end{proposition}
\begin{proof}
Let $(w_P,w_N) \in S(r,r) \times S(n-r,0),$ with $w_P=i_1 \cdots i_r|$ and $w_N=|j_1 \cdots j_{n-r},$ where $i_1, \cdots, i_r \in \{ \tilde{1}, \tilde{2}, \cdots \tilde{r}, 0^§\}$ and $j_1, \cdots ,j_{n-r} \in \{ 0^§, \overline{1},\overline{2}, \cdots, \overline{n-r} \}.$ \\
We set $\varphi (w_P,w_N)=i_1 \cdots i_r | j_1 \cdots j_{n-r}.$ It is easy to verity that $\varphi$ is an isomorphism between $S(r,r) \times S(n-r,0)$ and $S(n,r).$
\end{proof}
If we don't want to specify which elements of a string $w$ are in $\{ \tilde{1}, \cdots, \tilde{r},  0^§\}$ and which are in $\{
0^§, \overline{1}, \cdots \overline{n-r} \},$ we simply write $w=l_1 \cdots l_n,$ without specifying which $l_i$'s are in $\{ \tilde{1}, \cdots, \tilde{r}, 0^§\}$ and which are in $\{ 0^§, \overline{1}, \cdots \overline{n-r} \}.$ In any case, the order will be $l_1 \succeq l_2 \succeq \cdots \succeq l_n.$ \\

If $l,q \in A(n,r),$ we will set
$$
\delta (l,q) = \left\{ \begin{array}{lll}
                        \emptyset \,\,\, \textrm{if} \,\,\, l=q \\
                        (l,q) \,\,\, \textrm{if} \,\,\, l \neq q.
                        \end{array}
                        \right.
$$
If $w=l_1 \cdots l_n$ and $w'=l_1^{'} \cdots l_n^{'}$ are two strings in $S(n,r)$ we will set
$$
\Delta(w,w')=(\delta(l_1,l^{'}_1), \cdots, \delta(l_n,l^{'}_n)).
$$
\begin{proposition} \label{cover}
Let $w=l_1 \cdots l_n$ and $w'=l_1^{'} \cdots l_n^{'}$ be two strings in $S(n,r).$ Then:
$$
  w \models w'  \Longleftrightarrow
$$
$$
 \Delta(w,w')=(\emptyset, \cdots, \emptyset,(l_k,l'_k), \emptyset , \cdots ,\emptyset) \,\,\, \textrm{for some} \,\,\, k \in \{ 1, \cdots, n \} \,\,\, \textrm{where} \,\,\,l_k \vdash l'_k .
$$
\end{proposition}
\begin{proof}
$\Rightarrow$ By contradiction, we distinguish three cases:

1) there exists a couple $(l_k,l'_k)$ different from $\emptyset$ in $\Delta(w,w')$ such that $l_k \nvdash l'_k$. Since by hypothesis $w \models w'$, we have $w \sqsubset w'$; therefore it must be $l_k \prec l'_k$ and, for some $l \in A(n,r),$ it holds
$$l_k \prec l \prec l'_k.
$$

Hence the string $w_l=l_1 \cdots l_{k-1} l l_{k+1} \cdots l_n$ is such that $w \sqsubset w_l \sqsubset w',$ against the hypothesis.

2) there exist at least two couples $(l_k, l'_k),$ $(l_s,l'_s)$ with $(s>k)$ different from $\emptyset$ in $\Delta(w,w'),$ such that $l_k \vdash l'_k$  and $l_s \vdash l'_s$. Then, if we consider the string:
$$
u=l_1 \cdots l_{k-1}l_k l_{k+1} \cdots l_{s-1}l^{'}_s l_{s+1} \cdots l_n,
$$
it follows that $w \sqsubset u \sqsubset w',$ against the hypothesis.

3) all the components of $\Delta(w,w')$ are equal to $\emptyset.$ In this case, by definition of $\Delta(w,w')$ we will have that $w=w',$ against the hypothesis.

$\Leftarrow$ By hypothesis, we have that $w \sqsubseteq w'$ because  $\Delta(w,w')=(\emptyset, \cdots, \emptyset,(l_k,l'_k), \emptyset , \cdots ,\emptyset)$ with $l_k \vdash l'_k.$ Suppose that the thesis is false, then there exists a $w'' \in S(n,r)$ such that $w \sqsubset w'' \sqsubset w'.$ Let $w''=l_1^{''} \cdots l_n^{''}.$ By hypothesis it follows that
$$
l_i=l_i^{''}=l_i^{'}
$$
if $i\neq k,$ and $l_k \prec l_k^{''} \prec l_k^{'},$
and hence $l_k \nvdash l_k^{'},$ against the hypothesis.
\end{proof}
\medskip
We define now the function $\rho: S(n,r) \to \mathbb{N}_0$ as follows:
\\ if $w=i_1 \cdots i_r|j_1 \cdots j_{n-r} \in S(n,r)$ and we consider the symbols $i_1, \cdots,i_r, j_1, \cdots, j_{n-r}$ as non-negative integers (without $\tilde{}$ and $\bar{}$ ), then we set:
$$
\rho (w)= i_1 + \cdots + i_{r} + |j_1 -1| + \cdots + |j_{n-r} -(n-r)|=i_1+ \cdots + i_r + (1-j_1)+ \cdots + ((n-r)-j_{n-r}).
$$
\begin{proposition} \label{cover2}
The function $\rho$ satisfies the following two properties:
\begin{itemize}
\item[i)] $\rho(\hat{0})=0;$ \\
\item[ii)] if $w,w' \in S(n,r)$ and $w \models w'$, then $\rho(w')=\rho(w)+1.$
\end{itemize}
\end{proposition}
\begin{proof}
$i)$ Since $\hat{0}=0 \cdots 0|12\cdots (n-r),$ the thesis follows by the definition of $\rho.$

$ii)$ If $(w=l_1 \cdots l_n) \models (w'=l_1^{'} \cdots l_n^{'})$, by Proposition \ref{cover} we have that $\Delta(w,w')=(\emptyset, \cdots, \emptyset, (l_t,l'_t), \emptyset, \cdots \emptyset)$, for some $t \in \{ 1, \cdots, n \}$, with $l_t \vdash l'_t$. We distinguish different cases:

1) suppose that $1 \le t \le r$ and $l_t \succ 0^§.$ In this case we have that

$w=l_1 \cdots l_{t-1} l_t l_{t+1} \cdots l_r | l_{r+1} \cdots l_n$ and $w'=l_1 \cdots l_{t-1} (l_t +1) l_{t+1} \cdots l_r | l_{r+1} \cdots l_n$.

Hence $\rho (w')=l_1 + \cdots + l_{t-1} + (l_t +1) + l_{t+1} + \cdots + l_r + \sum_{k=1}^{n-r} (k-l_{r+k})=\sum_{k=1}^{r} l_k+\sum_{k=1}^{n-r}(k-l_{r+k})+1=\rho(w)+1.$

2) Suppose that $1 \le t \le r$ and $l_t=0^§ .$ In this case we have that $l'_t=1,$ hence $w=l_1 \cdots l_{t-1} 00 \cdots 0|l_{r+1} \cdots l_n$ and $w'=l_1 \cdots l_{t-1} 1 0 \cdots 0|l_{r+1} \cdots l_n,$ from which it holds that $\rho(w')=\rho(w) +1.$ \\
3) Suppose that $(r+1) \le t \le n$ and that $l_t=0^§.$ In this case we have a contradiction because there doesn't exist an element $l'_t$ in $\{ 0^§, \overline{1}, \cdots, \overline{n-r} \}$ which covers $0^§.$

4) Suppose that $(r+1) \le t \le n$ and that $l_t \prec 0^§;$ since we consider $l_t$ as an integer, it means that $1 \le l_t \le (n-r).$ In this case we have that $w=l_1 \cdots l_r | l_{r+1} \cdots l_{t-1} l_t l_{t+1} \cdots l_n$ and $w'=l_1 \cdots l_r | l_{r+1} \cdots l_{t-1} (l_t -1) l_{t+1} \cdots l_n.$ Therefore $\rho(w')= \sum_{i=1}^{r} l_i + \sum_{k=1,k\neq t-r}^{n-r} (k-l_{r+k})+[(t-r)-(l_t -1)]=\sum_{i=1}^r l_i + \sum_{k=1}^{n-r} (k-l_{r+k})+1=\rho(w)+1.$

\end{proof}

\begin{proposition}
$S(n,r)$ is a graded lattice having rank $R(n,r) = \binom{r+1}{2} + \binom{n-r+1}{2}$ and its rank function coincides with $\rho$.
\end{proposition}
\begin{proof}
A finite distributive lattice is also graded, (see \cite{stanley-vol1}), therefore $S(n,r)$ is graded by Proposition \ref{distributS(n,r)}.
In order to calculate the rank of $S(n,r),$ we need to determine a maximal chain and its length.
We consider the following chain $C$ in $S(n,r):$
$$
\begin{array}{ll}
\hat{1}=r (r-1) \cdots 2 1|0 0 \cdots 0 & \textrm{1 string with r elements}\\
 & \textrm{different from 0 on the left of}\,\,\, | \\
\end{array}
$$
$$
\vdots
$$
$$
\left\{\begin{array}{ll}
r (r-1) (r-2) 0 \cdots 0|0 0 \cdots 0 & \textrm{(r-2) strings with 3 elements} \\
\vdots & \textrm{different from 0 on the left of}\,\,\, | \\
r(r-1) 1 0 \cdots 0|0 0\cdots 0 &
\end{array}
\right.
$$
$$
\left\{ \begin{array}{ll}
r(r-1) 0 \cdots 0|0 0 \cdots 0 & \textrm{(r-1) strings with 2 elements} \\
\vdots & \textrm{different from 0 on the left of}\,\,\, | \\
r 1 0 \cdots 0|0 0 \cdots 0 &
\end{array}
\right.
$$
$$
\left\{ \begin{array}{ll}
r 0 \cdots 0|0 0 \cdots 0 & \textrm{r strings with 1 element} \\
\vdots & \textrm{different from 0 on the left of}\,\,\, | \\
1 0 \cdots 0|0 0 \cdots 0 &
\end{array}
\right.
$$

$$
0 0 \cdots 0|0 0 \cdots 0
$$

$$
\left\{\begin{array}{ll}
0 0 \cdots 0|0 0 \cdots 1 & \textrm{(n-r) strings with 1 element} \\
\vdots & \textrm{different from 0 on the right of}\,\,\, | \\
0 0 \cdots 0|0 0 \cdots (n-r) &
\end{array}
\right.
$$
$$
\vdots
$$
$$
\left\{\begin{array}{ll}
0 0 \cdots 0|0 0 1 (4 \cdots (n-r)) & \textrm{3 strings with (n-r-2) elements} \\
0 0 \cdots 0|0 0 2 (4 \cdots (n-r)) & \textrm{different from 0 on the right of}\,\,\, | \\
0 0 \cdots 0|0 0 3 (4 \cdots (n-r)) &
\end{array}
\right.
$$
$$
\left\{\begin{array}{ll}
0 0 \cdots 0|0 1 (3 4 \cdots (n-r)) & \textrm{2 strings with (n-r-1) elements} \\
0 0 \cdots 0|0 2 (3 4 \cdots (n-r)) & \textrm{different from 0 on the right of}\,\,\, |
\end{array}
\right.
$$
$$
\begin{array}{ll}
\hat{0}= 0 0 \cdots 0 0|(1 2 \cdots (n-r)) & \textrm{1 string with (n-r) elements} \\
& \textrm{different from 0 on the right of}\,\,\, | \\
\end{array}
$$
Therefore $C$ has exactly $(1+2+ \cdots + (r-1)+r) +1 + (1+2+ \cdots + (n-r))=\frac{r(r+1)}{2} + \frac{(n-r+1)(n-r)}{2} + 1= \binom{r+1}{2} + \binom{n-r+1}{2} +1$ elements and hence the length of $C$ is $\binom{r+1}{2} + \binom{n-r+1}{2} +1.$
By Proposition \ref{cover}, each element of the chain covers the previous one with respect to the order $\sqsubseteq$ in $S(n,r).$ Furthermore, $C$ has minimal element $\hat{0}$ (the minimum of $S(n,r)$) and maximal element $\hat{1}$ (the maximum of $S(n,r)$), hence $C$ is a maximal chain in $S(n,r).$ Since $C$ has $R(n,r)+1$ elements and $S(n,r)$ is graded, it follows that $S(n,r)$ has rank $R(n,r)$. Finally, since $S(n,r)$ is a graded lattice of rank $R(n,r)$ and it has $\hat{0}$ as minimal element, its rank function has to be the unique function defined on $S(n,r)$ and with values in $\{ 0,1, \cdots R(n,r) \}$ which satisfies the $i)$ and $ii)$ of Proposition \ref{cover2} (see \cite{stanley-vol1}). Hence such a function coincides with $\rho,$ by the uniqueness property.
\end{proof}

The following proposition shows that $w$ and $w^c$ are symmetric in the Hasse diagram of $S(n,r).$
\begin{proposition}\label{symmetryCompl}
If $w \in S(n,r),$ then $\rho(w)+\rho(w^c)=R(n,r).$
\end{proposition}
\begin{proof}
Let $w=i_1 \cdots i_r|j_1 \cdots j_{n-r}$ and $w^c=i_1^{'} \cdots i_r^{'}|j_1^{'} \cdots j_{n-r}^{'}.$ Then
$\rho(w)+\rho(w^c)=\sum_{k=1}^r i_k + \sum_{k=1}^{n-r} (k-j_k) + \sum_{k=1}^r i'_k + \sum_{k=1}^{n-r}(k-j_k^{'}).$
By definition of $w^c$ in $S(n,r),$ it follows that $\sum_{k=1}^r i_k + \sum_{k=1}^r i_k^{'} =\sum_{k=1}^r k$ and $\sum_{k=1}^{n-r} j_k + \sum_{k=1}^{n-r} j_k^{'} =\sum_{k=1}^{n-r} k$. Hence $\rho(w) + \rho(w^c)= \sum_{k=1}^r k + 2\sum_{k=1}^{n-r} k - \sum_{k=1}^{n-r} k=\sum_{k=1}^r k +\sum_{k=1}^{n-r} k =\binom{r+1}{2} +\binom{n-r+1}{2}=R(n,r)$.
\end{proof}

\section{The order reversing property of $\sqsubseteq$}\label{Sec3}

In general, $(S(n,r), \sqsubseteq, \hat{0},\hat{1})$ isn't a boolean lattice. For example,
if we take $w=54210|012 \in S(8,5),$ it is easy to verify that there doesn't exist an element $w' \in S(8,5)$ such that $w\wedge w'=\hat{0}$ and $w \vee w'=\hat{1}.$ In this section we will prove that the function
$w \in S(n,r) \mapsto w^c \in S(n,r)$ is {\it order reversing} with respect to the order $\sqsubseteq,$ in the sense that if $w_1 \sqsubseteq w_2,$ then $w_2^c \sqsubseteq w_1^c.$ In the sequel, we will see that this property is fundamental to prove many results of this paper.


\begin{proposition} \label{covcomp}
Let $w,w' \in S(n,r)$ be such that $w' \models w$  then  $w^c \models (w')^c.$
\end{proposition}
\begin{proof}
$\bf{\textrm{Case} \,\,\, 1)}$\\
Let $w$ and $w'$ be distinct in the following way:
$$
w=i_1 \cdots i_{r-s-1} 1 0 \cdots 0| \cdots
$$
$$
w'=i_1 \cdots i_{r-s-1} 0 0 \cdots 0| \cdots ,
$$
where $i_1 \succ \cdots \succ i_{r-s-1} \succ \tilde{1}.$\\
Consider now $(w^*)^{\pi}$ and $((w')^*)^{\pi}$ : they are two elements of $\mathcal{P}(A(n,r) \setminus \{ 0^§ \}).$\\
In $(w^*)^{\pi}$ there are $(s)$ elements of $A(n,r)$ $\succ 0^§$ and in $((w')^*)^{\pi}$ there are $(s+1)$ elements of $A(n,r)$ $\succ  0^§.$ Furthermore, $\tilde{1} \in ((w')^*)^{\pi}$ and $\tilde{1} \notin (w^*)^{\pi},$ hence the symmetric difference between $(w^*)^{\pi}$ and $((w')^*)^{\pi}$ is equal to $\{ \tilde{1} \}.$ From this, it follows that:
$$
(w')^c=\overline{((w')^*)^{\pi}}=t_1 \cdots t_s 1 0 \cdots 0| \cdots
$$
$$
w^c=\overline{(w^*)^{\pi}}=t_1 \cdots t_s 0 0 \cdots 0| \cdots,
$$
where $\{t_1, \cdots, t_s, \tilde{1} \}=\{ i_1, \cdots , i_{r-s-1},\overline{1}, \cdots, \overline{n-r} \}^c$ in $A(n,r) \setminus \{ 0^§ \}$ and $t_1 \succ \cdots \succ t_s \succ \tilde{1}.$ By Proposition \ref{cover}, it follows that $(w')^c$ covers $w^c.$ \\
$\bf{\textrm{Case} \,\,\, 2)}$\\
Adaptation of Case 1) to the elements on the right of $|$.\\
$\bf{\textrm{Case} \,\,\, 3)}$ \\
Let $k$ be the index in which $w$ and $w'$ are distinct, $1 \le k \le r,$, then:
$$
w = i_1 \cdots i_{k-1} (i_k +1) i_{k+1} \cdots i_p 0 \cdots 0| \cdots
$$
$$
w' = i_1 \cdots i_{k-1} i_k i_{k+1} \cdots i_p 0 \cdots 0| \cdots,
$$
with $i_1 \succ \cdots \succ i_{k-1} \succ i_k +1 \succ i_k \succ i_{k+1} \succ \cdots \succ i_p \succ  0^§.$ \\
Then, in $(w^*)^{\pi}$ there are exactly $q=(r-p)$ elements $\succ  0^§$ and in $((w')^*)^{\pi}$ there are also $q=(r-p)$ elements $\succ  0^§.$ Moreover, it follows that $i_k \in (w^*)^{\pi} \setminus ((w')^*)^{\pi}$ and $i_{k+1} \in ((w')^*)^{\pi} \setminus (w^*)^{\pi},$ hence the symmetric difference between $(w^*)^{\pi}$ and $((w')^*)^{\pi}$ is equal to $\{ i_k, i_{k+1} \}.$ From this, it follows that:
$$
(w')^c=\overline{((w')^*)^{\pi}}: t_1 \cdots t_{l-1} (i_{k} +1) t_{l+1} \cdots t_q 0 \cdots 0 | \cdots
$$
$$
w^c=\overline{(w^*)^{\pi}}: t_1 \cdots t_{m-1} i_k t_{m+1} \cdots t_q 0 \cdots 0| \cdots,
$$
where $i_{k} +1$ appears in the $l$-th place ($1 \le l \le r$) in $(w')^c$ and $i_k$ appears in the $m$-th place ($1 \le m \le r$)
in $w^c,$ with
\begin{equation} \label{thets}
\{ t_1, \cdots t_{l-1},t_{l+1}, \cdots ,t_{q} \}=\{t_1, \cdots t_{m-1}, t_{m+1}, \cdots, t_q \},
\end{equation}
where $\{ t_1, \cdots ,t_q \}= \{ i_1, \cdots i_k, i_k +1, i_{k+1}, \cdots, i_p, \overline{1}, \cdots \overline{n-r} \}^{\pi}$ in $A(n,r) \setminus \{ 0^§ \}.$ We prove now that the $l$-th place coincides with the $m$-th place.
Let $t \in \{ t_{l+1}, \cdots ,t_q \}$ and suppose by contradiction that $t \notin \{ t_{m+1}, \cdots, t_q \}.$
By (\ref{thets}) it follows that $t \in \{ t_1, \cdots ,t_{m-1} \}$, hence we will have $i_{k} + 1 \succ t$ and $t \succ i_k$,
and hence $i_k +1 \succ t \succ i_k$ in $A(n,r)$ and this contradicts $i_k \vdash (i_k +1)$.\\
Let now $t \in \{ t_1, \cdots, t_{m-1} \}.$ Suppose by contradiction that $t \notin \{ t_1, \cdots, t_{l-1} \}.$ By (\ref{thets}) it follows that $t \in \{ t_{l+1}, \cdots ,t_q \},$ hence we will have $t \succ i_k$ and $i_{k} + 1 \succ t,$ by which $i_{k} + 1 \succ t \succ i_k,$ and this contradicts $i_k \vdash i_k +1$. By (\ref{thets}) hence follows that $m=l$, and this proves that $w^c \models (w')^c$.\\
$\bf{\textrm{Case} \,\,\, 4)}$ \\
Analogously to Case $3)$ with $k$ such that $r+1 \le k \le n.$
\end{proof}

\begin{proposition} \label{complementary}
If $w,w' \in S(n,r)$ are such that $w' \sqsubseteq w,$ then $w^c \sqsubseteq (w')^c.$
\end{proposition}
\begin{proof}
It is enough to consider a sequence of elements $w_0,w_1,\cdots, w_n$ such that $w'=w_0 \sqsubseteq w_1 \sqsubseteq \cdots \sqsubseteq w_{n-1} \sqsubseteq w_n =w$ where $w_i$ covers $w_{i-1}$ for $i=1, \cdots, n$ and apply Proposition \ref{covcomp}
to $w_{i-1} \models w_{i}.$
\end{proof}
In general, a poset $\mathbb{P}=(P, \le)$ is called an {\it involution poset} if there exists a map $':P \to P$ such that $(i)\,\,\,(x')'=x$ and $(ii) \,\,\,x\le y,$ then $y' \le x'$ for all $x,y \in P.$ Recent studies related to this particular class of posets can be find in \cite{agha-gree-2001} and in \cite{brenneman-altri}. Hence by Proposition \ref{complementary} and Remark \ref{InvolutionPropertyS(n,r)}, $(S(n,r), \sqsubseteq, ^c,\hat{0},\hat{1})$ is an involution and distributive bounded lattice.
If $w=i_1 \cdots i_r|j_1 \cdots j_{n-r}$ is an element of $S(n,r),$ with $0 \le r \le n$, we can also consider the symbols $i_1, \cdots i_r,j_1, \cdots ,j_{n-r}$ as elements in the alphabet $A(n,n-r)$, where $j_1, \cdots ,j_{n-r} \in \{ \widetilde{n-r} \succ \cdots \succ \tilde{1} \succ 0^§ \}$ and $i_1, \cdots, i_r \in \{ 0^§ \succ \overline{1} \succ \cdots \succ \overline{r} \}$; in such case
we will set $w^t=j_{n-r} \cdots j_1 | i_r \cdots i_1.$ Then it holds that the map $w \in S(n,r) \mapsto w^t \in S(n,n-r)$ is bijective and it is such that
\begin{equation} \label{invinc}
w \sqsubseteq w' \,\,\, \textrm{in} \,\,\, S(n,r) \Longleftrightarrow (w')^t \sqsubseteq w^t \,\,\, \textrm{in} \,\,\, S(n,n-r),
\end{equation}
since $(w^t)^t=w,$ for each $w \in S(n,r)$.
 Also the map $w \in S(n,r) \mapsto w^c \in S(n,r)$ is bijective, and since $(w^c)^c=w,$ by Proposition \ref{complementary}, it follows that
\begin{equation} \label{invinc2}
w \sqsubseteq w' \,\,\, \textrm{in} \,\,\, S(n,r) \Longleftrightarrow (w')^c \sqsubseteq w^c \,\,\, \textrm{in} \,\,\, S(n,r).
\end{equation}
Therefore it holds the following isomorphism of lattices:
\begin{proposition} \label{prop8}
If $0 \le r \le n,$ then $S(n,r) \cong S(n,n-r).$
\end{proposition}
\begin{proof}
It is enough to consider the map $\varphi : S(n,r) \to S(n,n-r)$ defined by $\varphi(w)=(w^t)^c.$ Since the map $\varphi$ is the composition of the map $w \in S(n,r) \mapsto w^t \in S(n,n-r)$ with the map $u \in S(n,n-r) \mapsto u^c \in S(n,n-r),$ it follows that $\varphi$ is bijective. Furthermore, by (\ref{invinc}) and (\ref{invinc2}), it holds that
$$
w \sqsubseteq w' \,\,\, \textrm{in} \,\,\, S(n,r) \Longleftrightarrow \varphi (w) \sqsubseteq \varphi (w') \,\,\, \textrm{in} \,\,\, S(n,n-r).
$$
Hence $\varphi$ is an isomorphism of lattices.
\end{proof}
\begin{example}
$S(3,1) \cong S(3,2)$ and for example $\varphi(0|01)=((0|01)^t)^c=(10|0)^c=(20|1),$ or $\varphi(1|02)=((1|02)^t)^c=(20|1)^c=(10|0),$ see the Hasse diagrams of Section \ref{Sec2}.
\end{example}

\section{An Algorithmic Method for generating $S(n,r)$} \label{Sec4}

In this section we describe a generating algorithm for $S(n,r),$ which will permit us to fix an order, from the left to the right on each subset of the lattice composed by elements with fixed rank. In the Hasse diagram we will provide an algorithm for giving, on each line, a total order from the left to the right.\\
Set $w=l_1 \cdots l_n \in S(n,r)$ and let $k \in \{ 1, \cdots, n \}$ be fixed. If there exists an element $l_k^{'} \in A(n,r)$ which covers $l_k$ with respect to the order $\succ$ and such that $(l_1, \cdots, l_{k-1},l_k^{'}, l_{k+1}, \cdots l_n) \in S(n,r),$ we will say that $k$ is a {\it generating index} for the string $w.$ If $k$ is a generating index for $w$ and if $l_k^{'}$ is an element of $A(n,r)$ which covers $l_k,$ the string $(l_1, \cdots l_{k-1} l_k^{'} l_{k+1} \cdots l_n)$ will be called {\it string of index $k$ generated by $w$} and it will be denoted with the symbol $w[k].$ If $k$ is a generating index of $w$ contained in $\{ 1, \cdots , r \},$ we will say that $k$ is a {\it positive generating index} of $w;$ if $k$ is a generating index of $w$ contained in $\{ r+1, \cdots, n \},$ we will say that $k$ is a {\it negative generating index} of $w.$ Let now $s_1, \cdots, s_p$ be the positive generating indexes of $w$ (if $p=0$ there are no positive generating indexes of $w$) and $t_1, \cdots t_q$ the negative generating indexes of $w$ (if $q=0$ there are no negative generating indexes of $w$), with $s_1 < \cdots < s_p < t_1 < \cdots < t_q.$ \\
On the set $\{ w[s_1], \cdots, w[s_p], w[t_1], \cdots ,w[t_q] \}$ we introduce the following formal order $\lessdot:$
\begin{equation} \label{eq9}
w[s_1] \lessdot w[s_2] \lessdot \cdots \lessdot w[s_p] \lessdot w[t_q] \lessdot \cdots \lessdot w[t_1]
\end{equation}
In the Hasse diagram of $S(n,r)$ we will write the string generated by $w$ following the order given in $(\ref{eq9})$ :
$w[s_1]$ on the left of $w[s_2],$ $\cdots$ $,w[s_p]$ on the left of $w[t_q],$ $\cdots$, $w[t_2]$ on the left of $w[t_1]$.

\begin{example}
If $n=9,$ $r=5$ and $w=52000|0024,$ the positive generating indexes of $w$ are $2$ and $3$, while the negative generating indexes are $8$ and $9$, therefore, by (\ref{eq9}), we write $w[2]=53000|0024 \lessdot w[3]=52100|0024 \lessdot w[9]=52000|0023 \lessdot w[8]=52000|0014$.
\end{example}

Let now $k \in \{0,1, \cdots, R(n,r) \}$ be fixed. We denote $S_k(n,r)$ the set of elements of $S(n,r)$
with constant rank $k.$ We want to define a total order $\leftharpoondown$ on $S_k(n,r).$ If $k=0$ there is nothing to say because there is a unique element of rank $0.$ If $k=1,$ $S_1 (n,r)$ coincides with the set of strings generated by $0 \cdots 0|12 \cdots (n-r)$ and, in this case, $\leftharpoondown$ will coincide with the order $\lessdot$ given in (\ref{eq9}). \\
Let now $k$ be an integer such that $1 \le k < R(n,r)$ and suppose to have ordered with the total order $\leftharpoondown$ all the strings of $S_k(n,r).$ Suppose that $S_k(n,r)=\{ w_1, \cdots, w_m \}$ and that $w_1 \leftharpoondown w_2 \leftharpoondown \cdots \leftharpoondown w_m$ (in the Hasse diagram of $S(n,r)$ this implies that $w_1, \cdots ,w_m$ are written from the left to the right). \\
Let $w_i^1, \cdots , w_i^{k_i}$ be the strings of $S(n,r)$ generated by $w_i,$ for $i=1, \cdots, m.$ By (\ref{eq9}),
we can suppose that
\begin{equation*}
w_1^1 \lessdot \cdots \lessdot w_1^{k_1}, \cdots ,w_m^1 \lessdot \cdots \lessdot w_m^{k_m}.
\end{equation*}
We construct now $\leftharpoondown$ as follows:\\
at first we set
\begin{equation} \label{eq10}
w_1^1\leftharpoondown \cdots \leftharpoondown w_1^{k_1} \leftharpoondown w_2^1 \leftharpoondown \cdots \leftharpoondown w_2^{k_2} \leftharpoondown \cdots \leftharpoondown w_m^1 \leftharpoondown \cdots \leftharpoondown w_m^{k_m}.
\end{equation}
Then we have to eliminate in (\ref{eq10}) all the repeated strings.
\\ We examine the sequence (\ref{eq10}) starting from $w_1^1$ and continuing to the right till to $w_m^{k_m}.$ For $i=1, \cdots ,m$ if the string $w_i^{k_i}$ already appears among the strings on its left, then it will be deleted from the list (\ref{eq10}), otherwise it stays. At the end of the process it will remain the strings of $S_{k+1} (n,r),$ each one appearing only one time in the list (\ref{eq10}).

Let us observe that we have chosen to order the strings generated by $w$ as in (\ref{eq9}) because this choice gives great emphasis to the partition of $S(n,r)$ into two sublattices that we will describe in the next section; however we can also choose a different order with respect to  (\ref{eq9}), in fact, in some cases it is more useful to consider the following order on the subset of strings generated by $w$ :

\begin{equation} \label{eq9->}
w[s_1] \eqslantless w[s_2] \eqslantless \cdots \eqslantless w[s_p] \eqslantless w[t_1] \eqslantless \cdots \eqslantless w[t_q].
\end{equation}

In any case, no matter what is the chosen order, (\ref{eq9}) or (\ref{eq9->}), for the subset of the strings generated by $w$, the previous algorithm stays unchanged in all the other aspects. We will say that the previous generative algorithm for $S(n,r)$ is of type $\leftrightarrows$ if it is based on the order (\ref{eq9}), and of type $\rightrightarrows$ if it is based on the order (\ref{eq9->}). In the sequel of this paper we use the generative algorithm $\leftrightarrows$.

\begin{example}
Let $n=6,$ $r=3$ and $k=3.$ Then we have that $$S_0 (6,3)=\{000|123\}.$$
The generating indexes of $000|123$ are $1$ and $4$, therefore, by (\ref{eq9}), $$000|123[1]=100|123 \lessdot 000|123[4]= 000|023;$$
 hence
$$S_1 (6,3)= \{ 100|123 \leftharpoondown  000|023\}.$$
The generating indexes of $100|123$ are $1$ and $4$, therefore, as above, $$100|123[1]=200|123 \lessdot 100|123[4]= 100|023,$$
the generating indexes of $000|023$ are $1$ and $5$: $$000|023[1]=100|023 \lessdot 000|123[5]= 000|013;$$
hence (after having deleted the repeated strings)
$$S_2 (6,3)= \{ 200|123 \leftharpoondown  100|023 \leftharpoondown  000|013 \}.$$
The generating indexes of $200|123$ are $1$, $2$ and $4$: $$200|123[1]=300|123 \lessdot 200|123[2]= 210|023 \lessdot 200|123[4]=200|023,$$
the generating indexes of $100|023$ are $1$ and $5$: $$100|023[1]=200|023 \lessdot 100|023[5]= 100|013,$$
the generating indexes of $000|013$ are $1$, $5$ and $6$: $$000|013[1]=100|013 \lessdot 000|013[6]= 000|012 \lessdot 000|013[5]=000|003;$$
hence (after having deleted the repeated strings)
$$S_3 (6,3)= \{ 300|123 \leftharpoondown  210|123 \leftharpoondown  200|023 \leftharpoondown  100|013 \leftharpoondown  000|012 \leftharpoondown  000|003\}.$$

\end{example}

 In the next figure we have drawn the complete Hasse diagram of the lattice $S(6,3)$ in which each horizontal line represents the sub-poset $S_k (6,3)$ of the elements of rank $k,$ with $0 \le k \le 12,$ written in a totally ordered way from the left to the right following the total order $\leftharpoondown$ previously described.

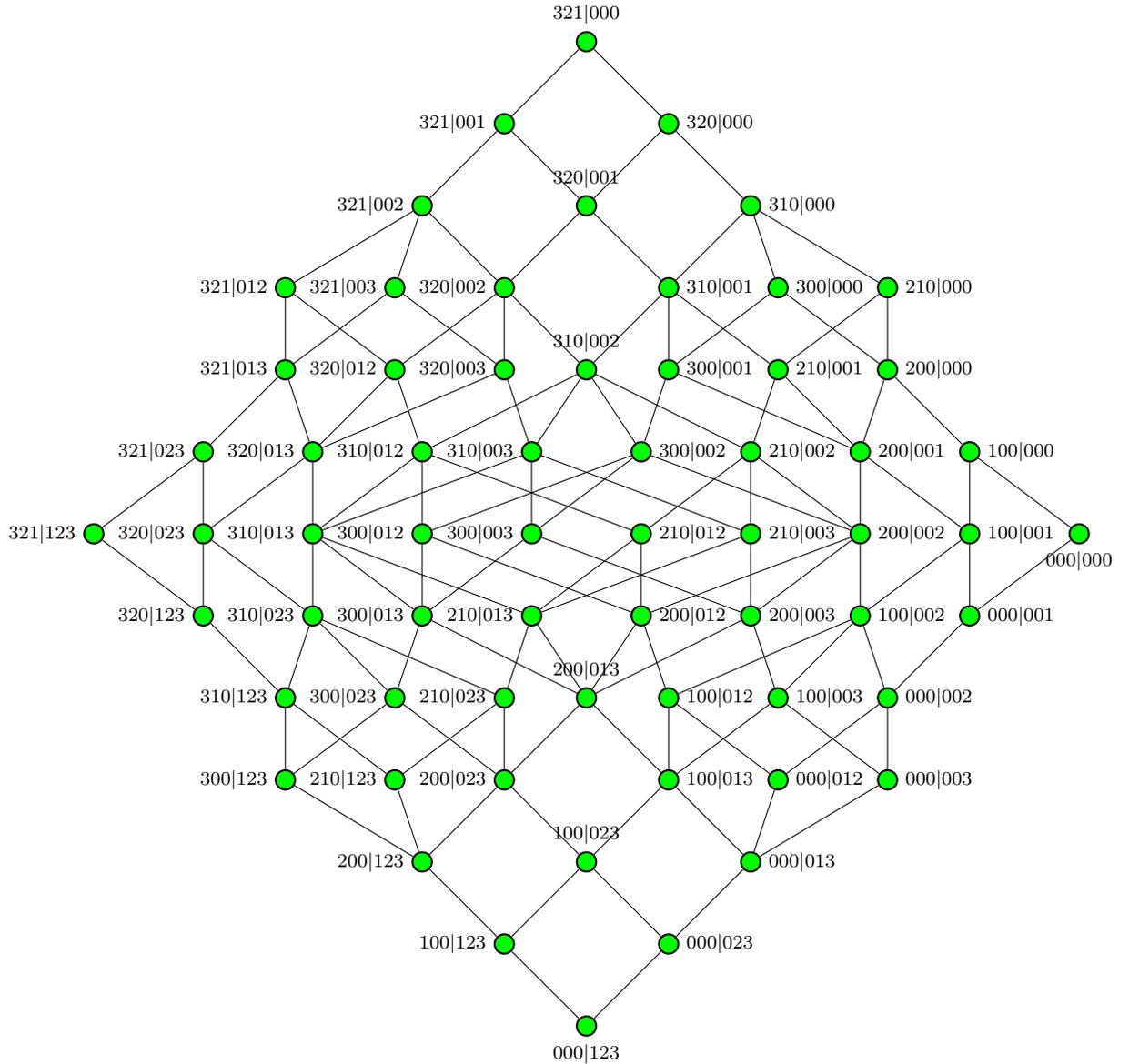
\begin{center}

\begin{tikzpicture}
 [inner sep=1.0mm,
 place/.style={circle,draw=black!100,fill=green!100,thick},scale=0.4]

 \path
 (-24,0)node(1a) [place,label=270:{\footnotesize$000|123$}]{}

 (-27,3)node(1b) [place,label=180:{\footnotesize$100|123$}]{}
 (-21,3)node(2b) [place,label=0:{\footnotesize$000|023$}]{}

 (-30,6)node(1c) [place,label=180:{\footnotesize$200|123$}]{}
 (-24,6)node(2c) [place,label=90:{\footnotesize$100|023$}]{}
 (-18,6)node(3c) [place,label=0:{\footnotesize$000|013$}]{}

 (-35,9)node(1d) [place,label=180:{\footnotesize$300|123$}]{}
 (-31,9)node(2d) [place,label=180:{\footnotesize$210|123$}]{}
 (-27,9)node(3d) [place,label=180:{\footnotesize$200|023$}]{}
 (-21,9)node(4d) [place,label=0:{\footnotesize$100|013$}]{}
 (-17,9)node(5d) [place,label=0:{\footnotesize$000|012$}]{}
 (-13,9)node(6d) [place,label=0:{\footnotesize$000|003$}]{}

 (-35,12)node(1e) [place,label=180:{\footnotesize$310|123$}]{}
 (-31,12)node(2e) [place,label=180:{\footnotesize$300|023$}]{}
 (-27,12)node(3e) [place,label=180:{\footnotesize$210|023$}]{}
 (-24,12)node(4e) [place,label=90:{\footnotesize$200|013$}]{}
 (-21,12)node(5e) [place,label=0:{\footnotesize$100|012$}]{}
 (-17,12)node(6e) [place,label=0:{\footnotesize$100|003$}]{}
 (-13,12)node(7e) [place,label=0:{\footnotesize$000|002$}]{}

 (-38,15)node(1f) [place,label=180:{\footnotesize$320|123$}]{}
 (-34,15)node(2f) [place,label=180:{\footnotesize$310|023$}]{}
 (-30,15)node(3f) [place,label=180:{\footnotesize$300|013$}]{}
 (-26,15)node(4f) [place,label=180:{\footnotesize$210|013$}]{}
 (-22,15)node(5f) [place,label=0:{\footnotesize$200|012$}]{}
 (-18,15)node(6f) [place,label=0:{\footnotesize$200|003$}]{}
 (-14,15)node(7f) [place,label=0:{\footnotesize$100|002$}]{}
 (-10,15)node(8f) [place,label=0:{\footnotesize$000|001$}]{}

 (-42,18)node(1g) [place,label=180:{\footnotesize$321|123$}]{}
 (-38,18)node(2g) [place,label=180:{\footnotesize$320|023$}]{}
 (-34,18)node(3g) [place,label=180:{\footnotesize$310|013$}]{}
 (-30,18)node(4g) [place,label=180:{\footnotesize$300|012$}]{}
 (-26,18)node(5g) [place,label=180:{\footnotesize$300|003$}]{}
 (-22,18)node(6g) [place,label=0:{\footnotesize$210|012$}]{}
 (-18,18)node(7g) [place,label=0:{\footnotesize$210|003$}]{}
 (-14,18)node(8g) [place,label=0:{\footnotesize$200|002$}]{}
 (-10,18)node(9g) [place,label=0:{\footnotesize$100|001$}]{}
 (-6,18)node(10g) [place,label=270:{\footnotesize$000|000$}]{}

 (-38,21)node(1h) [place,label=180:{\footnotesize$321|023$}]{}
 (-34,21)node(2h) [place,label=180:{\footnotesize$320|013$}]{}
 (-30,21)node(3h) [place,label=180:{\footnotesize$310|012$}]{}
 (-26,21)node(4h) [place,label=180:{\footnotesize$310|003$}]{}
 (-22,21)node(5h) [place,label=0:{\footnotesize$300|002$}]{}
 (-18,21)node(6h) [place,label=0:{\footnotesize$210|002$}]{}
 (-14,21)node(7h) [place,label=0:{\footnotesize$200|001$}]{}
 (-10,21)node(8h) [place,label=0:{\footnotesize$100|000$}]{}

 (-35,24)node(1i) [place,label=180:{\footnotesize$321|013$}]{}
 (-31,24)node(2i) [place,label=180:{\footnotesize$320|012$}]{}
 (-27,24)node(3i) [place,label=180:{\footnotesize$320|003$}]{}
 (-24,24)node(4i) [place,label=90:{\footnotesize$310|002$}]{}
 (-21,24)node(5i) [place,label=0:{\footnotesize$300|001$}]{}
 (-17,24)node(6i) [place,label=0:{\footnotesize$210|001$}]{}
 (-13,24)node(7i) [place,label=0:{\footnotesize$200|000$}]{}

 (-35,27)node(1j) [place,label=180:{\footnotesize$321|012$}]{}
 (-31,27)node(2j) [place,label=180:{\footnotesize$321|003$}]{}
 (-27,27)node(3j) [place,label=180:{\footnotesize$320|002$}]{}
 (-21,27)node(4j) [place,label=0:{\footnotesize$310|001$}]{}
 (-17,27)node(5j) [place,label=0:{\footnotesize$300|000$}]{}
 (-13,27)node(6j) [place,label=0:{\footnotesize$210|000$}]{}

 (-30,30)node(1k) [place,label=180:{\footnotesize$321|002$}]{}
 (-24,30)node(2k) [place,label=90:{\footnotesize$320|001$}]{}
 (-18,30)node(3k) [place,label=0:{\footnotesize$310|000$}]{}

 (-27,33)node(1l) [place,label=180:{\footnotesize$321|001$}]{}
 (-21,33)node(2l) [place,label=0:{\footnotesize$320|000$}]{}

 (-24,36)node(1m) [place,label=90:{\footnotesize$321|000$}]{};

 \draw (1a)--(1b); 
 \draw (1a)--(2b);

 \draw (1b)--(1c); 
 \draw (1b)--(2c);
 \draw (2b)--(2c);
 \draw (2b)--(3c);

 \draw (1c)--(1d); 
 \draw (1c)--(2d);
 \draw (1c)--(3d);
 \draw (2c)--(3d);
 \draw (2c)--(4d);
 \draw (3c)--(4d);
 \draw (3c)--(5d);
 \draw (3c)--(6d);

 \draw (1d)--(1e); 
 \draw (1d)--(2e);
 \draw (2d)--(1e);
 \draw (2d)--(3e);
 \draw (3d)--(2e);
 \draw (3d)--(3e);
 \draw (3d)--(4e);
 \draw (4d)--(4e);
 \draw (4d)--(5e);
 \draw (4d)--(6e);
 \draw (5d)--(5e);
 \draw (5d)--(7e);
 \draw (6d)--(6e);
 \draw (6d)--(7e);

 \draw (1e)--(1f); 
 \draw (1e)--(2f);
 \draw (2e)--(2f);
 \draw (2e)--(3f);
 \draw (3e)--(2f);
 \draw (3e)--(4f);
 \draw (4e)--(3f);
 \draw (4e)--(4f);
 \draw (4e)--(5f);
 \draw (4e)--(6f);
 \draw (5e)--(5f);
 \draw (5e)--(7f);
 \draw (6e)--(6f);
 \draw (6e)--(7f);
 \draw (7e)--(7f);
 \draw (7e)--(8f);

 \draw (1f)--(1g); 
 \draw (1f)--(2g);
 \draw (2f)--(2g);
 \draw (2f)--(3g);
 \draw (3f)--(3g);
 \draw (3f)--(4g);
 \draw (3f)--(5g);
 \draw (4f)--(3g);
 \draw (4f)--(6g);
 \draw (4f)--(7g);
 \draw (5f)--(4g);
 \draw (5f)--(6g);
 \draw (5f)--(8g);
 \draw (6f)--(5g);
 \draw (6f)--(7g);
 \draw (6f)--(8g);
 \draw (7f)--(8g);
 \draw (7f)--(9g);
 \draw (8f)--(9g);
 \draw (8f)--(10g);

 \draw (1g)--(1h); 
 \draw (2g)--(1h);
 \draw (2g)--(2h);
 \draw (3g)--(2h);
 \draw (3g)--(3h);
 \draw (3g)--(4h);
 \draw (4g)--(3h);
 \draw (4g)--(5h);
 \draw (5g)--(4h);
 \draw (5g)--(5h);
 \draw (6g)--(3h);
 \draw (6g)--(6h);
 \draw (7g)--(4h);
 \draw (7g)--(6h);
 \draw (8g)--(5h);
 \draw (8g)--(6h);
 \draw (8g)--(7h);
 \draw (9g)--(7h);
 \draw (9g)--(8h);
 \draw (10g)--(8h);

 \draw (1h)--(1i); 
 \draw (2h)--(1i);
 \draw (2h)--(2i);
 \draw (2h)--(3i);
 \draw (3h)--(2i);
 \draw (3h)--(4i);
 \draw (4h)--(3i);
 \draw (4h)--(4i);
 \draw (5h)--(4i);
 \draw (5h)--(5i);
 \draw (6h)--(4i);
 \draw (6h)--(6i);
 \draw (7h)--(5i);
 \draw (7h)--(6i);
 \draw (7h)--(7i);
 \draw (8h)--(7i);

 \draw (1i)--(1j); 
 \draw (1i)--(2j);
 \draw (2i)--(1j);
 \draw (2i)--(3j);
 \draw (3i)--(2j);
 \draw (3i)--(3j);
 \draw (4i)--(3j);
 \draw (4i)--(4j);
 \draw (5i)--(4j);
 \draw (5i)--(5j);
 \draw (6i)--(4j);
 \draw (6i)--(6j);
 \draw (7i)--(5j);
 \draw (7i)--(6j);

 \draw (1j)--(1k); 
 \draw (2j)--(1k);
 \draw (3j)--(1k);
 \draw (3j)--(2k);
 \draw (4j)--(2k);
 \draw (4j)--(3k);
 \draw (5j)--(3k);
 \draw (6j)--(3k);

 \draw (1k)--(1l); 
 \draw (2k)--(1l);
 \draw (2k)--(2l);
 \draw (3k)--(2l);

 \draw (1l)--(1m); 
 \draw (2l)--(1m);

 \end{tikzpicture}

\end{center}
\newpage

\section{A Recursive Formula for the Number \\of Elements in $S(n,r)$ of Rank $k.$} \label{Sec5}
In this section we give a recursive formula which counts the number of elements in $S(n,r)$ having fixed rank. At first we show that $S(n,r)$ can be seen as a translate union of two copies of $S(n-1,r)$ if $0 \le r < n$ and of $S(n-1,n-1)$ if $r=n$.

\begin{proposition} \label{prop9}
Let $n \ge 1$ and  $r\in \mathbb{N}$ such that $0 \le r \le n,$. Then there exist two disjoint sublattices $S_1 (n,r)$, $S_2 (n,r)$ of $S(n,r)$ such that $S(n,r)= S_1 (n,r) \cup S_2 (n,r)$, where:
\begin{itemize}
\item[i)] $S_i (n,r)\cong S(n-1, r)$ for $i=1,2,$ if $0 \le r < n$;\\
\item[ii)] $S_i (n,n)\cong S(n-1, n-1)$ for $i=1,2,$ if $r=n$. \\
\end{itemize}
\end{proposition}
\begin{proof}
We distinguish two cases: \\
$i)$ $0 \le r < n;$
we denote by $S_1 (n,r)$ the subset of $S(n,r)$ of all the strings of the form
$ w=i_1 \cdots i_r | j_1 \cdots j_{n-1-r} (n-r)$, with $j_1 \cdots j_{n-1-r} \in \{ 0^§, \overline{1}, \cdots, \overline{n-r-1} \}$; moreover, we denote by $S_2 (n,r)$ the subset of $S(n,r)$ of all the strings of the form
$ w=i_1 \cdots i_r |0 j_2 \cdots j_{n-r}$, with $j_2 \cdots j_{n-r} \in \{ 0^§, \overline{1}, \cdots, \overline{n-r-1} \}$.

It is clear that $S(n,r)$ is a disjoint union of $S_1 (n,r)$ and $S_2 (n,r).$ We prove now that $S_i (n,r) \cong S(n-1,r)$ for $i=1,2.$ Let $i=1$ (the case $i=2$ is analogous). It is obvious that there exists a bijective correspondence between $S_1 (n,r)$ and $S(n-1,r).$ Furthermore, if $w,w' \in S_1 (n,r)$ are such that $w=i_1 \cdots i_r | j_1, \cdots j_{n-r-1} (n-r),$ $w^{'}=i_1^{'} \cdots i_r^{'} |j_1^{'} \cdots j_{n-r-1}^{'} (n-r),$ it follows that $w \sqsubseteq w'$ (with respect to the order on $S(n,r)$) if and only if $i_1 \cdots i_r | j_1 \cdots j_{n-r-1} \sqsubseteq i_1^{'} \cdots i_r^{'} | j_1^{'} \cdots j_{n-r-1}^{'}$ (with respect to the order in $S(n-1,r)$).
Hence $S_1 (n,r)$ is isomorphic to $S((n-1),r).$ \\
Finally, since the order on $S(n,r)$ is component by component, it follows that each $S_i(n,r)$ (for $i=1,2$)
is a sublattice of $S(n,r).$\\
$ii)$ $r=n;$ by $i)$, there exist two disjoint sublattices $S_1 (n,0),$ $S_2 (n,0),$ of $S(n,0)$ such that $S(n,0)= S_1 (n,0) \cup S_2 (n,0)$, with $S_i (n,0) \cong S_i (n-1,0)$, for $i=1,2$. By Proposition \ref{prop8}, it follows that $S(n,n) \cong S(n,0)$, therefore there also exist two disjoint sublattices $S_1 (n,n),$ $S_2 (n,n),$ of $S(n,n)$ such that $S(n,n)= S_1 (n,n) \cup S_2 (n,n)$, where $S_i (n,n) \cong S_i (n,0) \cong S(n-1,0) \cong S(n-1,n-1)$, for $i=1,2,$ again by Proposition \ref{prop8}.
\end{proof}

If $n\ge 1,$ the element of minimal rank of the sublattice $S_2(n,r)$ is obviuosly $\hat{w} = 0 \cdots 0 |012 \cdots (n-r-1).$ \\
This element has rank $0$ as element of $(S_2(n,r), \sqsubseteq ),$ but in $S(n,r)$ has rank given by

$\rho (\hat{w})=(1-0)+(2-1)+(3-2)+\cdots +((n-r)-(n-r-1)) = n-r$.

Therefore we can visualize $S_2(n,r)$ (in the Hasse diagram of $S(n,r)$) as an upper-translation of the sublattice $S_1 (n,r),$ of height $(n-r).$ \\
\begin{example}
For example, this is the Hasse diagram of $S(5,3)$ as a translate union of $S_1 (5,3) \cong S(4,3)$ (red lattice) and of
$S_2 (5,3) \cong S(4,3)$ (green lattice).
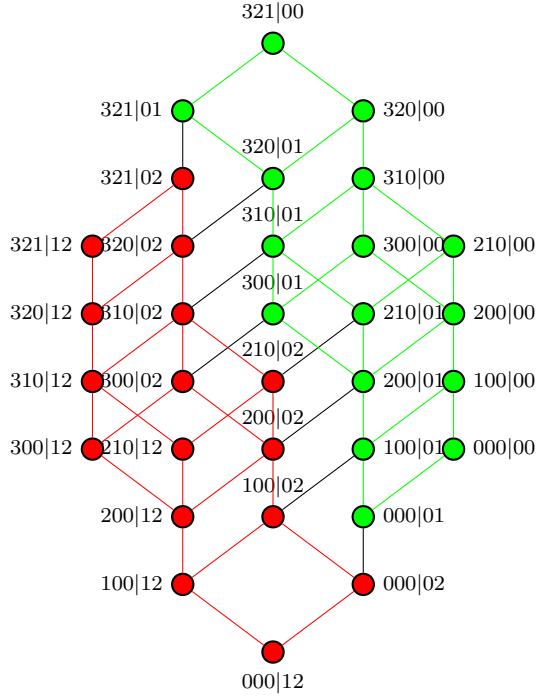
\begin{center}
\begin{tikzpicture}
 [inner sep=1.0mm,
 placeg/.style={circle,draw=black!100,fill=green!100,thick},
 placer/.style={circle,draw=black!100,fill=red!100,thick},
  scale=0.3]

 \path
 (0,0)node(1a) [placer,label=270:{\footnotesize$000|12$}]{}

 (-4,3)node(1b) [placer,label=180:{\footnotesize$100|12$}]{}
 (4,3)node(2b) [placer,label=0:{\footnotesize$000|02$}]{}

 (-4,6)node(1c) [placer,label=180:{\footnotesize$200|12$}]{}
 (0,6)node(2c) [placer,label=90:{\footnotesize$100|02$}]{}
 (4,6)node(3c) [placeg,label=0:{\footnotesize$000|01$}]{}

 (-8,9)node(1d) [placer,label=180:{\footnotesize$300|12$}]{}
 (-4,9)node(2d) [placer,label=180:{\footnotesize$210|12$}]{}
 (0,9)node(3d) [placer,label=90:{\footnotesize$200|02$}]{}
 (4,9)node(4d) [placeg,label=0:{\footnotesize$100|01$}]{}
 (8,9)node(5d) [placeg,label=0:{\footnotesize$000|00$}]{}

 (-8,12)node(1e) [placer,label=180:{\footnotesize$310|12$}]{}
 (-4,12)node(2e) [placer,label=180:{\footnotesize$300|02$}]{}
 (0,12)node(3e) [placer,label=90:{\footnotesize$210|02$}]{}
 (4,12)node(4e) [placeg,label=0:{\footnotesize$200|01$}]{}
 (8,12)node(5e) [placeg,label=0:{\footnotesize$100|00$}]{}

 (-8,15)node(1f) [placer,label=180:{\footnotesize$320|12$}]{}
 (-4,15)node(2f) [placer,label=180:{\footnotesize$310|02$}]{}
 (0,15)node(3f) [placeg,label=90:{\footnotesize$300|01$}]{}
 (4,15)node(4f) [placeg,label=0:{\footnotesize$210|01$}]{}
 (8,15)node(5f) [placeg,label=0:{\footnotesize$200|00$}]{}

 (-8,18)node(1g) [placer,label=180:{\footnotesize$321|12$}]{}
 (-4,18)node(2g) [placer,label=180:{\footnotesize$320|02$}]{}
 (0,18)node(3g) [placeg,label=90:{\footnotesize$310|01$}]{}
 (4,18)node(4g) [placeg,label=0:{\footnotesize$300|00$}]{}
 (8,18)node(5g) [placeg,label=0:{\footnotesize$210|00$}]{}

  (-4,21)node(1h) [placer,label=180:{\footnotesize$321|02$}]{}
 (0,21)node(2h) [placeg,label=90:{\footnotesize$320|01$}]{}
 (4,21)node(3h) [placeg,label=0:{\footnotesize$310|00$}]{}

 (-4,24)node(1i) [placeg,label=180:{\footnotesize$321|01$}]{}
 (4,24)node(2i) [placeg,label=0:{\footnotesize$320|00$}]{}

 (0,27)node(1j) [placeg,label=90:{\footnotesize$321|00$}]{};

 \draw[red] (1a)--(1b); 
 \draw[red] (1a)--(2b);

 \draw[red] (1b)--(1c); 
 \draw[red] (1b)--(2c);
 \draw[red] (2b)--(2c);
 \draw (2b)--(3c);

 \draw[red] (1c)--(1d); 
 \draw[red] (1c)--(2d);
 \draw[red] (1c)--(3d);
 \draw[red] (2c)--(3d);
 \draw (2c)--(4d);
 \draw[green] (3c)--(4d);
 \draw[green] (3c)--(5d);

 \draw[red] (1d)--(1e); 
 \draw[red] (1d)--(2e);
 \draw[red] (2d)--(1e);
 \draw[red] (2d)--(3e);
 \draw[red] (3d)--(2e);
 \draw[red] (3d)--(3e);
 \draw (3d)--(4e);
 \draw[green] (4d)--(4e);
 \draw[green] (4d)--(5e);
 \draw[green] (5d)--(5e);

 \draw[red] (1e)--(1f); 
 \draw[red] (1e)--(2f);
 \draw[red] (2e)--(2f);
 \draw (2e)--(3f);
 \draw[red] (3e)--(2f);
 \draw (3e)--(4f);
 \draw[green] (4e)--(3f);
 \draw[green] (4e)--(4f);
 \draw[green] (4e)--(5f);
 \draw[green] (5e)--(5f);

 \draw[red] (1f)--(1g); 
 \draw[red] (1f)--(2g);
 \draw[red] (2f)--(2g);
 \draw (2f)--(3g);
 \draw[green] (3f)--(3g);
 \draw[green] (3f)--(4g);
 \draw[green] (4f)--(3g);
 \draw[green] (4f)--(5g);
 \draw[green] (5f)--(4g);
 \draw[green] (5f)--(5g);

 \draw[red] (1g)--(1h); 
 \draw[red] (2g)--(1h);
 \draw (2g)--(2h);
 \draw[green] (3g)--(2h);
 \draw[green] (3g)--(3h);
 \draw[green] (4g)--(3h);
 \draw[green] (5g)--(3h);

 \draw (1h)--(1i); 
 \draw[green] (2h)--(1i);
 \draw[green] (2h)--(2i);
 \draw[green] (3h)--(2i);

 \draw[green] (1i)--(1j); 
 \draw[green] (2i)--(1j);

 \end{tikzpicture}

\end{center}

\end{example}
Given the lattice $S(n,r),$ for each $k$ such that $0 \le k \le R(n,r),$ we denote with $s(n,r,k)$ the number of elements of $S (n,r)$ with rank $k.$ It holds the following ricorsive formula for $s(n,r,k):$
\begin{proposition} \label{prop10}
Let $n \ge 1.$ If $r\in \mathbb{N}$ is such that $0 \le r < n,$ then
$$
s(n,r,k)= \left\{ \begin{array}{lll}
                  s(n-1,r,k) & \textrm{if} & 0 \le k < (n-r) \\
                  s(n-1,r,k) + s(n-1,r, k-(n-r)) & \textrm{if} & (n-r) \le k \le R(n-1,r) \\
                  s(n-1,r,k-(n-r)) & \textrm{if} & R(n-1,r) < k \le R(n,r)
                  \end{array} \right.
$$
If $r=n$, then $s(n,n,k)=s(n,0,k)$.
\end{proposition}
\begin{proof}
$\bf{\textrm{Case} \,\,\, 1)}$ Let $k$ be such that $0 \le k < (n-r).$ By what we have asserted before, the element $\hat{w}$ (i.e. the minimum of $S_2 (n,r)$) has rank $(n-r)$ in $S(n,r),$ hence, by Proposition \ref{prop9}, it follows that $s(n,r,k)$ coincides with the number of elements of rank $k$ in $S_1 (n,r)$ and since $S_1 (n,r) \cong S(n-1,r),$ it follows that $s(n,r,k)=s(n-1,r,k).$\\
$\bf{\textrm{Case} \,\,\, 2)}$ Let $k$ be such that $(n-r) \le k \le R(n-1,r).$ In this case, the number of elements of rank $k$ in $S(n,r)$ coincides with the sum of the number of elements of rank $k$ in $S_1 (n,r)$ and of the number of elements of rank $[k-(n-r)]$ in $S_2 (n,r).$ Since $S_1 (n,r) \cong S_2 (n,r) \cong S(n-1,r),$ it follows that $s(n,r,k)=s(n-1,r,k)+s(n-1,r,k-(n-r)).$\\
$\bf{\textrm{Case} \,\,\, 3)}$ Let $k$ be such that $R(n-1,r) < k \le R(n,r).$ In this case $s(n,r,k)$ coincides with the number of elements of rank $[k-(n-r)]$ in $S_2 (n,r),$ and since $S_2 (n,r) \cong S(n-1,r)$ it follows that $s(n,r,k)=s(n-1,r,k-(n-r)).$

Finally, if $r=n$, the last equality follows from the isomorphism $S(n,n) \cong S(n,0)$.
\end{proof}

It is clear that we would prefer a closed formula for the numbers $s(n,r,k)$, however at present the previous recursive formula is the best possible result that we have.
By the recursive formula stated in Proposition \ref{prop10}, the first values of $s(n,r,k)$ are given by:\\
$s(0,0,0)=1,$\\
$s(1,0,0)=s(0,0,0)=1,$ $\,\,\,\,\,$ $s(1,0,1)=s(0,0,0)=1,$ \\
$s(1,1,0)=s(1,0,0)=1,$ $\,\,\,\,\,$ $s(1,1,1)=s(1,0,1)=1,$ \\
$s(2,0,0)=1,$ $s(2,0,1)=s(1,0,1)=1,$ $s(2,0,2)=s(1,0,0)=1,$ \\
$s(2,0,3)=s(1,0,1)=1,$ \\
$s(2,1,0)=s(1,1,0)=1,$ $s(2,1,1)=s(1,1,1)+s(1,1,0)=2,$ $s(2,1,2)=s(1,1,1)=1,$ \\
$s(2,2,0)=s(2,0,0)=1,$ $s(2,2,1)=s(2,0,1)=1,$ $s(2,2,2)=s(2,0,2)=1,$ \\
$s(2,2,3)=s(2,0,3)=1.$ \\

If $P$ is graded poset of rank $m$ and has $p_{i}$ elements of rank $i$, for each $0\leq i\leq m$ , then the polynomial $F\left(P,t\right)=\sum_{i=1}^{m} p_{i}t^{i}$ is called the $rank$-$generating function$ of P.
If $P$ and $Q$ are two graded posets respectively with rank-generating functions $F\left(P,t\right)$ and $F\left(Q,t\right)$, then $P\times Q$ is also graded and $F\left(P\times Q,t\right)=F\left(P,t\right)\cdot F\left(Q,t\right)$ (see \cite{stanley-vol1}).\\
This leads to the following Cauchy-type formula for $s(n,r,k)$.
\

\begin{proposition}\label{s(n,r,k)Cauchy}
If$\,\,$ $0\leq r \leq n$ and $0\leq k\leq R(n,r)$ then $s(n,r,k)=\sum_{i=0}^{k}s(r,r,i)\cdot s(n-r,n-r,k-i)$.
\end{proposition}
\begin{proof}
The rank-generating function of $S(n,r)$ is $F\left(S(n,r),t\right)=\sum_{k=0}^{R(n,r)}s(n,r,k)t^{k}$.

By Propositions \ref{CartesianIsomorphism} and \ref{prop8} it follows that

$S(n,r) \cong S(r,r) \times S(n-r,0)\cong S(r,r) \times S(n-r,n-r).$

Hence

 $F\left(S(n,r),t\right)=F\left(S(r,r)\times S(n-r,n-r),t\right)=F\left(S(r,r),t\right)\cdot F\left(S(n-r,n-r),t\right).$

Then

 $F\left(S(n,r),t\right)=\left(\sum_{l=0}^{R(r,r)}s(r,r,l)t^{l}\right)\cdot\left(\sum_{j=0}^{R(n-r,n-r)}s(n-r,n-r,j) t^{j}\right)\\=\sum_{k=0}^{R(r,r)+R(n-r,n-r)}\sum_{i=0}^{k}s(r,r,i)s(r,r,k-i)t^{k}\\
=\sum_{k=0}^{R(n,r)}\sum_{i=0}^{k}s(r,r,i)s(n-r,n-r,k-i)t^{k},$

hence the thesis follows.
\end {proof}

The last result of this section shows a symmetric property of $S(n,r)$.

\begin{proposition}\label{s(n,r,k)-symm}
If $0\leq r \leq n$ and $k=R(n,r)$, then $s(n,r,i)=s(n,r,k-i)$ for $0 \le i \le k$.
\end{proposition}
\begin{proof}
We recall that $S_l(n,r)$ is the set of elements of $S(n,r)$ with rank $l$, for each $0\leq l\leq k$.
It is enough to consider the map $f:S_i(n,r)\rightarrow S_{k-i}(n,r)$ defined by $f(w)=w^{c}$.\\At first we observe that $f$ is well-defined, because if $w\in S_i(n,r)$ then $\rho(w)=i$ and, by Proposition \ref{symmetryCompl}, $\rho(w^{c})=k-i$, therefore $w^{c}\in S_{k-i}(n,r)$. The map
$f$ is injective because by $(w^{c})^{c}=w$ it follows that $w_{1}^{c}=w_{2}^{c}\ \Rightarrow w_{1}=w_{2}$. To show that $f$ is also onto, we take $v\in S_{k-i}(n,r)$ and $w=v^{c}$. Since $\rho(v)=k-i$, by Proposition \ref{symmetryCompl} we have that $k=\rho(v^{c})+\rho(v)=\rho(w)+\rho(v)=\rho(w)+(k-i) $, hence $ \rho(v^{c})=i$, i.e. $ w \in S_i(n,r)$ and $f(w)=w^{c}=(v^{c})^{c}=v$, so $f$ is onto and hence $f$ is bijective.
\end{proof}

\newpage

\section{Relation between Weight Functions, the Lattices $S(n,r)$ \\ and $S(n,d,r)$ and the Numbers $\gamma(n,r)$ and $\gamma(n,d,r).$} \label{Sec6}

\begin{definition}
A $(n,r)$-function is an application $f:A(n,r) \to \mathbb{R}$ which is increasing and such that $f(0^§)=0,$ i.e.:
\begin{equation}\label{equ1}
f(\tilde{r}) \ge \cdots \ge f(\tilde{1}) \ge f(0^§)=0 > f(\overline{1}) \ge \cdots \ge f(\overline{n-r}).
\end{equation}
We call $F(n,r)$ the set of the $(n,r)$-functions.
\end{definition}
\begin{definition}
The function $f$ is a $(n,r)$-weight function if (\ref{equ1}) holds and if:
\begin{equation}\label{eq2}
f(\tilde{1}) + \cdots + f(\tilde{r}) + f(\overline{1}) + \cdots + f(\overline{n-r}) \ge 0.
\end{equation}
We call $WF(n,r)$ the set of the $(n,r)$-weight functions.
\end{definition}
\begin{definition}
If $f$ is a $(n,r)$-function, we define the {\it sum function} induced by $f$ on $S(n,r)$
$$\Sigma_f : S(n,r) \to \mathbb{R}$$
the function that associates to $w \in S(n,r),$
$w=i_1 \cdots i_r \,\,\, |\,\,\, j_1 \cdots j_{n-r},$
the real number $\Sigma_f (w)=f(i_1)+ \cdots + f(i_r) + f(j_1) + \cdots + f(j_{n-r}).$
\end{definition}
\begin{proposition} \label{propa}
If $f$ is a $(n,r)$-function and if $w,w' \in S(n,r)$ are such that $w \sqsubseteq w',$ then $\Sigma_f (w) \le \Sigma_f (w').$
\end{proposition}
\begin{proof}
If $w=i_1 \cdots i_r \,\,\, | \,\,\, j_1 \cdots j_{n-r} \sqsubseteq w'=i_1^{'} \cdots i_r^{'} | j_1^{'} \cdots j_{n-r}^{'},$
then we have that $i_1 \preceq i_1^{'},$ $\cdots,$ $i_r \preceq i_r^{'},$ $j_1 \preceq j_1^{'},$ $\cdots$ $j_{n-r} \preceq j_{n-r}^{'};$ hence, since $f$ is increasing on $A(n,r),$ the assertion follows immediately by definition of the sum function $\Sigma_f.$
\end{proof}
\begin{proposition} \label{propb}
If $f$ is a $(n,r)$-weight function and if $w \in S(n,r)$ is such that $\Sigma_f (w)<0,$ then $\Sigma_f (w^c) >0.$
\end{proposition}
\begin{proof}
By definition of the two binary operations $\sqcap,$ $\sqcup$ and of the complement operation $^{c}$ on $S(n,r),$
we have that
$$ w \sqcup w^{c}=r \cdots 1| 1 \cdots (n-r)$$
and
$$ w \sqcap w^{c}=0 0 \cdots 0 | 0 \cdots 0.$$
Hence, by definition of $\Sigma_f$ and since $f(0^§)=0,$ we have that
$$ \Sigma_f (w) + \Sigma_f (w^c) =\Sigma_f (w \sqcup w^{c})=f(\tilde{1})+ \cdots + f(\tilde{r}) + f(\overline{1}) + \cdots + f(\overline{n-r}) \ge 0,$$
by (\ref{eq2}). Hence, if $\Sigma_f (w) < 0,$ we will have that $\Sigma_f (w^{c}) >0.$
\end{proof}
If $f$ is a $(n,r)$-function, we set:
$$S^+_f (n,r) = \{ w \in S(n,r) : \Sigma_f (w) \ge 0 \} ;$$
furthermore, if $d$ and $r$ are integers such that $1 \le d,r \le n,$ we set:
$$S^+_f (n,d,r) = \{ w \in S(n,d,r) : \Sigma_f (w) \ge 0 \}.$$
Observe that, in general, neither $S^{+} _f (n,r)$ nor $S^{+} _f (n,d,r)$ are sublattices of $S(n,r),$ because they are not closed with respect to the operation of $\inf$ ($\wedge$). They are simply sub-posets of $S(n,r)$ with the induced order.\\
We set
\begin{itemize}
\item[$(\beta) $] $\gamma (n,r) = \min \{ |S^+_f (n,r)|$ : $f$ is a $(n,r)$-weight function $\}$; \\
\item[$(\delta)$] $\gamma(n,d,r)= \min \{ | S^+_f (n,d,r)|$ : $f$ is a $(n,r)$-weight function $\}$. \\
\end{itemize}
It is easy to observe that the numbers defined in $(\beta)$ are exactly those in (\ref{eq0}) of the introduction, while the numbers defined in $(\delta)$ are the same of those in (\ref{eq1}) in the introduction. We use therefore both the notations.

The Theorem 1 of \cite{ManMik87} applied to our context, gives the following
\begin{proposition}
For each $r \in \mathbb{N}$ with $1 \le r \le n$, we have that:
\begin{itemize}
\item[i)]  $\gamma(n,r) \ge 2^{n-1}+1$, \\
\item[ii)] $\gamma (n,1) \le 2^{n-1} +1.$
\end{itemize}
\end{proposition}
The difference between our situation and the Theorem 1 of \cite{ManMik87} is that we admit the string $0 \cdots 0 | 0 \cdots 0$ in the set $S^+_f (n,r),$ i.e. we admit the empty set. For this reason in $i)$ and $ii)$ the number $(2^{n-1} +1)$ appears instead of $2^{n-1}$ of \cite{ManMik87}.\\
In the next section, we will link the numbers $\gamma(n,r)$ and $\gamma(n,d,r)$ to a minimum problem on a family of boolean functions defined on the lattices $S(n,r)$ and $S(n,d,r)$. 

\section{Weight Functions and Boolean Functions on $S(n,r).$} \label{Sec7}
In this section we show how to associate to any $(n,r)$-function and to any $(n,r)$-weight function a Boolean function on $S (n,r).$ Our aim is to connect the study of the $(n,r)$-weight functions and of the related extremal problems (in particular the computation of $\gamma (n,r)$ and $\gamma (n,d,r)$) to some boolean functions on $S(n,r).$ \\
If $f$ is a $(n,r)$-function (or a $(n,r)$-weight function), we can define the map
$$
A_f : S(n,r) \to \bf{2}
$$
setting
$$
A_f (w) = \left\{ \begin{array}{lll}
                  P & \textrm{if} & \Sigma_f (w) \ge 0 \\
                  N & \textrm{if} & \Sigma_f (w) < 0
                  \end{array} \right.
$$
In order to underline the essential properties of the map $A_f,$ we introduce the concept of $(n,r)$-boolean  map.\\

\begin{definition}
A $(n,r)$-boolean  map (briefly $(n,r)$-BM) is a map $A : S(n,r) \to \bf{2}$  with the following properties:
\begin{itemize}
\item[$a_1)$] if $w_1,w_2 \in S(n,r)$ and $w_1 \sqsubseteq w_2$, with $A (w_1)=P$, then $A (w_2)=P$;\\
\item[$a_2)$] if $w_1, w_2 \in S(n,r)$ and $w_1 \sqsubseteq w_2$, with  $A (w_2)=N$, then  $A (w_1)=N;$ \\
\item[$a_3)$]  $A(0 \cdots 0|0 \cdots 0)=P$ and $A(0 \cdots 0|0 \cdots 0 1)=N.$
\end{itemize}
\end{definition}
\begin{definition}
A $(n,r)$-{\it weighted boolean  map} (briefly $(n,r)$-WBM) is a $(n,r)$-BM $A : S(n,r) \to \bf{2}$ which satisfies the following two properties:
\medskip
\begin{itemize}
\item[$a_4)$] if $w \in S(n,r)$ is such that $A (w)=N,$ then  $A (w^{c}) =P;$ \\
\item[$a_5)$]  $A (r (r-1) \cdots 2 1 | 1 2 \cdots (n-r))=P.$
\end{itemize}
\end{definition}

We denote by $B(n,r)$ the family of all the $(n,r)$-BM's and by $WB(n,r)$ the family of all the $(n,r)$-WBM's.

\begin{proposition}\label{propd}
\begin{itemize}
\item[i)] If $A:S(n,r) \to \bf{2}$ then:
$$A\,\,\, \textrm{satisfies} \,\,\, a_1) \Longleftrightarrow A \,\,\, \textrm{satisfies} \,\,\, a_2) \Longleftrightarrow A \,\,\, \textrm{is order-preserving}.$$
\item[ii)] if $f$ is a $(n,r)$-function, then $A_f$ is a $(n,r)$-BM. \\
\item[iii)] if $f$ is a $(n,r)$-weight function, then $A_f$ is a $(n,r)$-WBM.\\
\end{itemize}
\end{proposition}
\begin{proof}

$i)$ The assertion is straightforward, thanks to an argument by contradiction.

$ii)$ Let $f$ be a $(n,r)$-function. Suppose that $w_1,w_2 \in S(n,r)$ and that $w_1 \sqsubseteq w_2.$ By Proposition \ref{propa}, it follows that $\Sigma_f(w_1) \le \Sigma_f(w_2).$ Suppose that $A_f(w_1)>A_f(w_2).$ This would imply that $A_f(w_1)=P$ and $A_f(w_2)=N,$ i.e. (by definition of $A_f$) $\Sigma_f(w_1) \ge 0$ and $\Sigma_f(w_2) <0$ and this is a contradiction. Hence $A_f$ is order-preserving. The property $a_3)$ holds by definition of $\Sigma_f$ and $A_f.$ \\
$iii)$ Let $f$ be a $(n,r)$-weight function. Let $w \in S(n,r)$ such that $A_f(w)=N.$ By definition of $A_f,$ we have that $\Sigma_f (w) <0,$ and hence, by Proposition \ref{propb}, $\Sigma_f(w^c) >0,$ i.e. $A_f(w^c)=P.$ Hence $A_f$ satisfies $a_4)$. The property $a_5)$ is obviuosly satisfied, by definition of $\Sigma_f$ and $A_f$ since $f$ is a $(n,r)$-weight function.
\end{proof}
\begin{definition}
A map $A \in B(n,r)$ is said to be {\it numerically represented} if there exists a $(n,r)$-function $f \in F(n,r)$ such that $A=A_f.$
\end{definition}
\begin{definition}
A map $A\in WB(n,r)$ is said to be {\it numerically represented} if there exists a $(n,r)$-weight function $f \in WF(n,r)$ such that $A=A_f.$
\end{definition}
Set
$$RB(n,r)=\{ A \in B(n,r): A \textrm{ is numerically represented} \},$$
$$RWB(n,r)=\{ A \in WB(n,r): A \textrm{ is numerically represented} \} .$$
\begin{proposition}
\begin{itemize}
\item[i)] $RB(n,r)$ is identified with a quotient of $F(n,r).$ \\
\item[ii)] $RWB(n,r)$ is identified with a quotient of $WF(n,r).$
\end{itemize}
\end{proposition}
\begin{proof}
$i)$ We define on $F(n,r)$ the following binary relation: if $f,g \in F(n,r),$ we set $f \sim g$ if for each $w \in S(n,r)$, we have that $\Sigma_f (w) \ge 0 \Leftrightarrow \Sigma_g (w) \ge 0$.

Then $\sim$ is an equivalence relation on $F(n,r).$ By Proposition \ref{propd}-$ii),$ if $f \in F(n,r)$ it follows that $A_f \in B(n,r).$ Therefore it is defined a map $\varphi: F(n,r) \to B(n,r)$ such that $\varphi(f)=A_f.$ \\
Then, if $f,g \in F(n,r),$ it follows that:
$$(f \sim g)\Longleftrightarrow (\textrm{for each}\,\, w \in S(n,r), \Sigma_f(w) \,\,\, \textrm{and}\,\,\, \Sigma_g(w) \,\,\, \textrm{have the same sign})$$
$$\Longleftrightarrow (A_f(w)=A_g(w) \,\,\, \textrm{for each} \,\,\, w \in S(n,r)) \Longleftrightarrow \varphi(f)=\varphi(g).$$
By the Universal Property of the quotient, there exists a unique injective map $\tilde{\varphi} : F(n,r)
\to B(n,r)$ such that the following diagram commutes:
\begin{equation*} 
\begin{CD}
\xymatrix{
 F(n,r) \ar[rrrr]^{\varphi} \ar[rrd]_{\nu} & & & & B(n,r) \\
& & F(n,r)/\sim \ar[rru]_{\tilde{\varphi}} & &  }
\end{CD}
\end{equation*}
where $\nu$ is the projection on the quotient. Since the image of $\varphi$ is exactly $RB(n,r)$ and since it coincides with the image of $\tilde{\varphi},$ it follows that $\tilde{\varphi}$ is a bijective map between $F(n,r)/\sim$ and $RB(n,r).$\\ Analogously we prove $ii),$ using Proposition \ref{propd}-$iii).$
\end{proof}
If $A \in WB(n,r),$ we will set
$$S_A ^+ (n,r) = \{ w \in S(n,r) : A(w)=P \}, $$
and if $d\ge 1$ is such that $d \le n,$ we will set:
$$S_A ^+ (n,d,r) = \{ w \in S(n,d,r) : A(w)=P \}. $$
Furthermore, we set:
$$\tilde{\gamma} (n,r) = \min \{ |S^+ _A (n,r)| : A \in WB(n,r) \}, $$
$$\tilde{\gamma} (n,d,r) = \min \{ |S^+ _A (n,d,r)| : A \in WB(n,r) \},$$
$$\overline{\gamma} (n,r) = \min \{ |S^+ _A (n,r)| : A \in RWB(n,r) \}, $$
$$\overline{\gamma} (n,d,r) = \min \{ |S^+ _A (n,d,r)| : A \in RWB(n,r) \}.$$
\begin{proposition} \label{propebis}
\begin{itemize}
\item[i)] $\gamma (n,r)=\overline{\gamma}(n,r) \ge \tilde{\gamma} (n,r).$ \\
\item[ii)] $\gamma (n,d,r)=\overline{\gamma}(n,d,r) \ge \tilde{\gamma} (n,d,r).$
\end{itemize}
\end{proposition}
\begin{proof}
i) The inequality $\overline{\gamma}(n,r) \ge \tilde{\gamma}(n,r)$ is obviuos because $RWB(n,r)$ is a subset of $WB(n,r).$
We prove that $\gamma(n,r)=\overline{\gamma} (n,r).$ Let $f$ be a $(n,r)$-weight function for which it holds $\gamma (n,r) =|S^+ _{f} (n,r)|.$ Then $A_f \in RWB(n,r)$ and (by definition of $A_f$) we have that: $S_{A_f} ^+ (n,r)  = \{ w \in S(n,r) : A_f (w)=P \} = \{ w \in S(n,r) : \Sigma_f (w) \ge 0 \}  =  S^+ _f (n,r).$

Hence
$$\gamma (n,r) = |S^+ _f (n,r)|=|S^+_{A_f}| \ge \overline{\gamma} (n,r),$$
because $A_f \in RWB(n,r).$ \\
On the contrary, let $A \in RWB(n,r)$ such that $\overline{\gamma} (n,r)=|S^+_A (n,r)|.$\\
Since $A$ is numerically represented, there will exist a $(n,r)$-weight function $f \in WF(n,r)$ such that $A=A_f.$ Then we have that
$$
S_{A}^+ (n,r)=S_{A_f}^+ (n,r)=S_{f}^+ (n,r),
$$
and hence
$$
\overline{\gamma} (n,r)=|S_{A} ^+ (n,r)|=|S_{f}^+ (n,r)| \ge \gamma(n,r).
$$
This proves that $\gamma(n,r)=\overline{\gamma}(n,r).$ \\
The proof of $ii)$ is similar to $i).$
\end{proof}
It is natural now to assert the following two problems:

{\bf First Open Problem : } $B(n,r)=RB(n,r)$ ? \\
{\bf Second Open Problem : } $WB(n,r)=RWB(n,r)$ ? \\

If $RWB(n,r)$ coincides with $WB(n,r)$ (i.e. if any $(n,r)$-weighted boolean  map is numerically represented) then $\gamma(n,r)=\tilde{\gamma}(n,r)$ and $\gamma(n,d,r)=\tilde{\gamma} (n,d,r),$ by Proposition \ref{propebis}. If the answer to the second open problem is affirmative, this would imply that each time we give the boolean formal values N or P to each string of $S (n,r)$ in such a way that the rules $a_1)-a_5)$ are respected, then there exists a numerical attribution to the singletons which permits the reconstruction of the configuration of N's and P's in a unique way. In other words, if the assertion of the Second Open Problem holds we have an effective representation theorem.

\end{document}